\documentclass[11pt]{amsart}

\usepackage{amssymb}

\usepackage[all]{xy}

\usepackage{xspace}		
\usepackage{amsthm}		
\usepackage{enumerate}		
\usepackage{graphicx}		
\usepackage{ifthen}			
\usepackage{mathrsfs}		
\usepackage{xargs}			

\newfont\eul{eusm10}
\newfont\euls{eusm8}

\newcommand{\bea}{\begin{eqnarray}}
\newcommand{\eea}{\end{eqnarray}}
\newcommand{\bean}{\begin{eqnarray*}}
\newcommand{\eean}{\end{eqnarray*}}

\newcommand{\mcM}{\mathcal{M}}
\newcommand{\mcR}{\mathcal{R}}
\newcommand{\mcS}{\mathcal{S}}
\newcommand{\mcT}{\mathcal{T}}
\newcommand{\mcU}{\mathcal{U}}
\newcommand{\mcV}{\mathcal{V}}

\newcommand{\mfk}{\mathfrak{k}}

\newcommand{\mbbR}{\mathbb{R}}

\newcommand{\mbfK}{\mathbf{K}}
\newcommand{\mbfL}{\mathbf{L}}
\newcommand{\mbfM}{\mathbf{M}}

\newcommand{\mbfV}{\mathbf{V}}

\newcommand{\Kc}{\mathbf{K}^c}

\newcommand{\On}{\mathbf{On}}

\DeclareMathOperator{\cof}{\mathsf{cf}}

\DeclareMathOperator{\col}{\mathsf{col}}

\DeclareMathOperator{\crp}{\mathsf{cr}}
\newcommand{\cut}{\mathop{||}}

\newcommand{\dom}{\mathrm{dom}}

\newcommand{\htp}{\mathsf{ht}}
\newcommand{\idm}{\mathsf{id}}

\newcommand{\lht}{\mathsf{lh}}

\newcommand{\otp}{\mathsf{otp}}

\newcommand{\ptm}{\mbox{\eul P}}

\newcommand{\res}{\mathop{|}}

\newcommand{\rng}{\mathrm{rng}}
\newcommand{\rst}{\restriction}

\newcommand{\stl}{<_{\mathsf{S}}}

\newcommand{\ult}{\mathsf{Ult}}
\newcommand{\wfp}{\mathsf{wfp}}
\newcommand{\zfc}{{\mathsf{ZFC}}}

\newcommand{\card}{\mathsf{card}}

\newcommand{\dfeq}{\mathop{\stackrel{\mathrm{def}}{=}}}

\newcommand{\trcl}{\mathsf{trcl}}

\newtheorem{lemma}{Lemma}

\newtheorem{theorem}[lemma]{Theorem}
\newtheorem{corollary}[lemma]{Corollary}
\newtheorem{fact}[lemma]{Fact}
\newtheorem{definition}[lemma]{Definition}
\newtheorem{remark}[lemma]{Remark}
\newtheorem{question}[lemma]{Question}
\newtheorem{conjecture}[lemma]{Conjecture}
\newtheorem{claim}[lemma]{Claim}

\newcommand{\twiddle}{\raisebox{1pt}{\scalebox{.75}{$\mathord{\sim}$}}}

\begin{document}


\begin{abstract}
If $\mbfM$ is a proper class inner model of $\zfc$ and
$\omega_2^{\mbfM}=\omega_2$, then every sound mouse projecting to $\omega$ and
not past $0^\P$ belongs to ${\mbfM}$. In fact, under the assumption  
that $0^\P$ does not belong to ${\mbfM}$, ${\mbfK}^{\mbfM}\|\omega_2$ is
universal for all countable mice in ${\mbfV}$.  

Similarly, if $\mbfM$ is a proper class inner model of $\zfc$,
$\delta>\omega_1$ is regular, $(\delta^+)^{\mbfM}= \delta^+$, and in ${\mbfV}$
there is no proper class inner model with a Woodin cardinal, then
${\mbfK}^{\mbfM}\|\delta$ is universal for all mice in ${\mbfV}$ of
cardinality less than $\delta$.  
\end{abstract}

\author[Andr\'es E. Caicedo]{Andr\'es Eduardo Caicedo}
\address{
Andr\'es E. Caicedo \\
Boise State University \\
Department of Mathematics\\ 
1910 University Drive\\
Boise, ID 83725
}
\email{caicedo@math.boisestate.edu}
\thanks{The first author was supported in part by NSF grant DMS-0801189.}
\urladdr{http://math.boisestate.edu/{\twiddle}caicedo}
\curraddr{Mathematical Reviews \\
416 Fourth Street \\
Ann Arbor, MI 48103-4820 \\
USA
}
\email{aec@ams.org}
\urladdr{http://www-personal.umich.edu/{\twiddle}caicedo/}

\author{Martin Zeman}
\address{
Martin Zeman\\
University of California at Irvine\\
Department of Mathematics\\
Irvine, CA 92697
}
\email{mzeman@math.uci.edu}
\thanks{The second author was supported in part by NSF grant DMS-0500799, and by a UCI CORCL Special Research Grant.}
\urladdr{http://www.math.uci.edu/{\twiddle}mzeman/}



\title[Universality of local core models]{Downward transference of mice and universality of local core models}

\maketitle

\tableofcontents


\section{Introduction} \label{sect:intro}

\subsection{Motivation}

This paper is the second in a series started with Caicedo \cite{cai10}.  The goal is to study the 
structure of (not necessarily fine structural) inner models of the set theoretic universe ${\mbfV}$, 
perhaps in the presence of additional axioms, assuming that there is agreement between (some of) the 
cardinals of ${\mbfV}$ and of the inner model. 

Although the results in this paper do not concern forcing axioms, it was a recently solved problem in the 
theory of the proper forcing axiom, ${\sf PFA}$, that motivated our results. 

It was shown in Veli\v ckovi\'c \cite{Velickovic} that if Martin's maximum, ${\sf MM}$, holds, and 
${\mbfM}$ is an inner model that computes ${\omega}_2$ correctly, then ${\mathbb R}\subset {\mbfM}$. 
The argument requires the weak reflection principle to hold, a consequence of both ${\sf MM}$ and 
${\sf PFA}^+$. Todor\v cevi\'c \cite{Todorcevic} showed that the same result holds if, rather than 
${\sf MM}$, the reflection principle known as Rado's conjecture is assumed.

It is natural to wonder whether ${\sf PFA}$ suffices as well. A partial result in this direction was shown 
in Caicedo-Veli\v ckovi\'c \cite{CaicedoVelickovic}, namely that if ${\mbfM}$ is an inner model that 
computes ${\omega}_2$ correctly, and {\em both} ${\mbfV}$ and ${\mbfM}$ satisfy the bounded proper 
forcing axiom, ${\sf BPFA}$, then again ${\mathbb R}\subset {\mbfM}$. 

However, it was recently shown by Friedman \cite{Friedman} (see also Neeman \cite{Neeman}) that the 
answer to the ${\sf PFA}$ question is negative: Using a variant of the method of models as side 
conditions, Friedman defines a forcing notion that, starting with a supercompact $\kappa$, produces an 
extension where ${\sf PFA}$ holds and $\kappa$ becomes $\omega_2$. The generic for the partial 
order where only the side conditions are considered also collapses $\kappa$ to size $\omega_2$, but 
the resulting extension does not contain all the reals. 

Assume that ${\sf PFA}$ holds in ${\mbfV}$, and that ${\mbfM}$ is an inner model with the same 
$\omega_2$. It seems not completely unreasonable to expect that, even though not all reals can be 
guaranteed to be in ${\mbfM}$,  the reals of ${\mbfV}\setminus {\mbfM}$ should be {\em generic} in 
some sense. A possible way of formalizing this intuition consists in trying to show that the large cardinal 
strength coded by reals in ${\mbfV}$ is also present in ${\mbfM}$. When approaching this problem, we 
realized that there may be some ${\sf ZFC}$ results that the presence of forcing axioms could be hiding. 

To illustrate this, we mention a couple of observations:
 \begin{enumerate}
\item
    Assume first that ${\mbfM}$ is an inner model such that ${\sf CAR}^{\mbfM}={\sf CAR}$, where 
    ${\sf CAR}$ denotes the class of cardinals, and that ${\sf PFA}$ holds. Then, for example, 
    ${\sf AD}^{{\mbfL}({\mathbb R})}$ holds in ${\mbfM}$ and in all its set generic extensions, by Steel 
    \cite{S}. The point here is that Steel \cite{S} shows that ${\sf AD}^{{\mbfL}({\mathbb R})}$ holds in any
    generic extension of the universe by a forcing of size at most $\kappa$, whenever 
    $\square_{\kappa}$ fails for ${\kappa}$ a strong limit singular cardinal. If ${\sf PFA}$ holds, the claim 
    follows by recalling that ${\sf PFA}$ contradicts $\square_{\kappa}$ for all uncountable ${\kappa},$ 
    and that $\square_{\kappa}$ relativizes up to any outer model where ${\kappa}$ and ${\kappa}^+$ 
    are still cardinals. 
In particular, many mice of $\mbfV$ are in $\mbfM$. For example, $0^\P$, the
 sharp for a model with a strong cardinal, as well as the sharp for a model
 with a proper class of strong cardinals, are both in $\mbfM$ by 
 absoluteness. (On the other hand, although $\mbfM$ thinks that $M_1^\sharp$
 exists, it is not immediately clear whether its version of $M_1^\sharp$ is the
 $M_1^\sharp$ of $\mbfV$.) 
  
\item
    If ${\sf PFA}$ holds and ${\mbfM}$ is an inner model that satisfies ${\sf BPFA}$ and computes 
    ${\omega}_2$
    correctly, then $H_{{\omega}_2}\subset {\mbfM}$ by Caicedo-Veli\v ckovi\'c \cite{CaicedoVelickovic}, 
    so the setting of Claverie-Schindler \cite{Schindler2} applies. Therefore, using notation as in 
    Claverie-Schindler \cite{Schindler2}, ${\mbfM}$ is closed under any mouse operator ${\mathcal M}$ 
    that does not go beyond $M_1^\sharp$.
\end{enumerate}

The moral here is that it seems that we only need to assume a local version of agreement of 
cardinals rather than a global one to conclude that an inner model ${\mbfM}$ must absorb any 
significant large cardinal strength coded by reals in ${\mbfV}$. 

\subsection{Main result}

It is interesting to note that if ${\mbfM}$ computes ${\omega}_2$ correctly, and $0^\sharp$ exists, 
then $0^\sharp\in {\mbfM}$, see Friedman \cite{F}. We have started a systematic approach to the 
question of how far this result can be generalized. Using techniques different
from those in Friedman 
\cite{F}, we reprove this result, and our method allows us to show that the same conclusion can be 
obtained with $0^\sharp$ replaced by stronger mice, see for example Lemma \ref{lem:Fried} and 
Corollary~\ref{cor:smallmice}.

\begin{conjecture} \label{3.3}
Let $r$ be a 1-small sound (iterable) mouse that projects to $\omega.$ Assume that ${\mbfM}$ is an 
inner model and that ${\omega}_2^{\mbfM}={\omega}_2$. Then $r\in {\mbfM}$.
\end{conjecture}

The expected argument to settle Conjecture \ref{3.3} is an induction on the mice hierarchy. The 
mice as in the conjecture are lined up so that, if it fails, then there is a least counterexample $r$. 
Comparing $r$ with ${\mbfK}^{\mbfM}$, one expects to reach the contradiction that $r$ is in fact an 
initial segment of ${\mbfK}^{\mbfM}$. Our result is that this is indeed the case under appropriate 
anti-large cardinal assumptions, see for instance Corollary 10. 

We will consider a proper class inner model $\mbfM$ and the true core
model of the universe $\mbfM$; we denote this core model by
$\mbfK^{\mbfM}$. Of course, we will only consider situations where
$\mbfK^{\mbfM}$ exists, so we need to assume that in $\mbfM$ there is no
proper class inner
model with a Woodin cardinal. We will actually make the assumption that there
is no proper class inner model with a Woodin cardinal in the sense of $\mbfV$.  
Then proper class extender models of $\mbfM$ fully iterable in the sense of
$\mbfM$ are fully iterable in
$\mbfV$, which will play an important role in the proof of our main
theorem. Here we mean iterability with respect to fine structural iteration
trees. Also, in this paper we consider the terms ``premouse" and ``extender
model" to be synonymous, although in the literature the term ``extender model"
often implicitly means that the model in question is a proper class. This
explains our use of the term  ``proper class extender model" above.  

The proof of this kind of absoluteness of iterability is standard,
and we give it in  Section~\ref{sect:cmt}. We briefly comment in Section
\ref{sec:dis} on what the  
situation is under the weaker hypothesis that there is no proper
class inner model with a Woodin cardinal in the sense of $\mbfM$.

We remark that, although we restrict ourselves to the 
situation where there are no Woodin cardinals in inner models, we expect that
our universality results can be strengthened to reach beyond this point. 

We will be using the core model theory as developed in Steel \cite{cmip}; this
approach uses 
the core model constructed in a universe below a measurable cardinal, which
simplifies the exposition. Thus the universe we will be working with will
have the property that its ordinals constitute a measurable cardinal in some
longer universe. We believe, but did not check carefully, that our
arguments for the core model up to a Woodin cardinal can be carried out -- in
a more or less straightforward way -- in the core model theory developed in Jensen-Steel \cite{js} which 
avoids the use of a measurable cardinal. We have tried to make our exposition as independent of
the direct reference to the existence of a measurable cardinal in $\mbfV$ as
possible, so that there will be enough heuristic evidence for our belief. 
Regarding the fine structural aspects, we will use Jensen's $\Sigma^*$-fine
structure theory as described in Zeman \cite{imlc} or Welch \cite{ssfs}, and use some
notation from Zeman \cite{imlc}; we believe this notation is
self-explanatory. Otherwise our notation will be fully compatible with that in
Steel \cite{cmip}. 

Following Steel \cite{cmip}, when working with the core model up to a Woodin cardinal, $\Omega$ will 
be reserved to denote the measurable/subtle cardinal
needed to develop the basic core model theory in the universe we are working
in. By $\mbfV$, we will mean the rank-initial segment of the
universe up to $\Omega$. The notion ``class" will then have the corresponding
meaning. When working with the core model up to a strong cardinal we let
$\Omega=\On$, that is, the class of all ordinals, let $\mbfV$ be the actual
universe, and use the core model theory below $0^\P$, the sharp for a strong cardinal, as described in 
Zeman \cite{imlc}. It should be stressed that the proof of our main theorem, Theorem~\ref{t0.1}, will
only require to construct $\mbfK^\mbfM$, the core model of $\mbfM$. In this
case the universe we will be working in will be $\mbfM$, so we will work
below $\Omega$ relative to $\mbfM$.   

Recall the following notation that will be used throughout the paper. We will
use Mitchell-Steel indexing of extenders as introduced in
Mitchell-Steel~\cite{fsit}, see also 
\cite{oimt}, with the only exception when we work with the core model below
$0^{\P}$, in which case we use Jensen's indexing as described in Chapter~8 in
Zeman~\cite{imlc}. We will make it explicit in the text when we use Jensen's
indexing. If $N$
is any extender model and $\alpha\le\htp(N)$ is an ordinal, then
$N\cut\alpha$ is the extender model
$\langle J^E_\alpha,E_{\omega\alpha}\rangle$ where $E$ is the extender
sequence of $N$. If $\alpha$ does not index an extender in $N$, for instance
when $\alpha$ is a cardinal in $N$, then $E_{\omega\alpha}=\varnothing$ and in
this case we identify $N\cut\alpha$ with $J^E_\alpha$. Recall also that if
$\mcT$ is a normal iteration tree then $\delta(\mcT)$ is the supremum of all
natural lengths of extenders used in the tree; see the beginning of
Section~\ref{sect:cmt} for more details. If $\mcT$ is an iteration tree of
successor length then the main branch of the tree is the branch connecting the
first model of the tree with its last model. 

The following lemma is a particular case of a well-known result of Shelah, see
for example Abraham \cite{fwch}. The point of this short detour is to explain
why we introduce the set $S_\delta$ below and  
concentrate on it rather than directly assuming agreement of cardinals, as
discussed above. We use the Foreman-Magidor definition of
$\ptm_\kappa(\lambda)$ from Foreman-Magidor~\cite{lcdcch}, so 
\begin{equation}\label{e.pkl}
\ptm_\kappa(\lambda)=\{x\subseteq\lambda\mid\otp(x)<\kappa\;\&\;x\cap\kappa
\mbox{ is transitive.}\}
\end{equation}
The notion of stationarity is that in Woodin's sense, that is, a set
$S\subseteq\ptm(Z)$ is stationary iff for every partial map 
$f:{}^{<\omega}Z\to Z$ there is some $x\in X$ which is closed under $f$. 

\begin{lemma} \label{lem:isstat}
Let ${\mbfM}$ be a proper class inner model of $\zfc$ such that
$\omega_2^{\mbfM}=\omega_2.$ Then $\ptm_{\omega_1}(\omega_2)\cap {\mbfM}$ is
stationary. 
\end{lemma}

\begin{proof}
Let $f:{}^{<\omega}\omega_2\to\omega_2,$ and pick $\gamma<\omega_2$ of size
$\omega_1$ which is closed under $f$. Since ${\mbfM}$ computes $\omega_2$
correctly, it sees a bijection $f:\omega_1\to\gamma$. It is
then easy to find an $\alpha<\omega_1$ such that $f[\alpha]$ is closed under
$f$ and $f[\alpha]\cap\omega_1=\alpha$.
\end{proof}

To illustrate how the assumption of stationarity of $\ptm_{\omega_1}(\omega_2)\cap {\mbfM}$ can be 
used, consider the following:

\begin{lemma} \label{lem:Fried}
Let ${\mbfM}$ be a proper class inner model of $\zfc$ such that
$\ptm_{\omega_1}(\omega_2)\cap {\mbfM}$ is stationary. Suppose that $0^\sharp$
exists. Then $0^\sharp\in {\mbfM}.$ 
\end{lemma}

\begin{proof}
Assume $P$ is the (unique, sound) mouse representing $0^\sharp$, so 
 $$ P=(J_\beta,\in,F), $$
where $F$ is an extender representing an amenable measure on an ordinal $\mu$
and $\beta$ is the Mitchell-Steel index of $F$, that is,
$\beta=\mu^{++\mbfL}$. (See the beginning of Section~\ref{sect:cmt} for more
details on Mitchell-Steel indexing.) Furthermore, $P$ has first projectum
$\varrho^1_P=\omega$. Let $\theta$ be regular and large. By assumption,
we can find some countable $X\preceq H_\theta$ such that:
\begin{enumerate}
\item
$P \in X$,
\item
$X \cap \omega_2 \in {\mbfM}$, and
\item
$\kappa=X \cap \omega_1 \in \omega_1$.
\end{enumerate}

Let $H$ be the transitive collapse of $X$, and let 
 $$ \sigma: H \to H_\theta $$ 
be the inverse of the collapsing map. Then the critical point of $\sigma$ is
$\kappa$, $\sigma(\kappa)=\omega_1$, and $P\in H$.

Let $P'$ be the $\kappa$-th iterate of $P$, formed by applying the ultrapower
construction using $F$ and its images under the
corresponding embeddings. Then $P'\in H$ and also
 $$ \ptm(\kappa)\cap {\mbfL} \in H\cap {\mbfL}. $$ 
Note that $P'=(J_{\beta'},F')$, 
where $\beta'$ is the double cardinal successor of $\crp(F')$ in ${\mbfL}$,
and that $\sigma\upharpoonright J_{\beta'}\in {\mbfM}$. 

By the argument from the proof of the comparison lemma,
 $$ \sigma\upharpoonright P' = \pi_{\kappa,\omega_1}, $$ 
where the map on the right side is the iteration map. Hence $F'$ is the
 ${\mbfL}$-extender on $\kappa$ derived from 
$\sigma\upharpoonright J_{\beta'}$. It follows that $F'\in\mbfM$. But then
 $P\in\mbfM$. 
\end{proof}

Of course, Lemma \ref{lem:isstat} generalizes straightforwardly to larger
cardinals, which suggests that  
rather than looking at direct generalizations of Lemma~\ref{lem:Fried} to
stronger countable mice, one  
should instead focus on instances of {\em local universality}. We will work
only with $1$-small premice/extender models, and such premice/extender models
have at most one iteration strategy, where branches for trees of limit length 
are given by $Q$-structures which are levels of $\mbfL$ (see Steel~\cite{oimt}
for the notion of $Q$-structure). Such a $Q$-structure for $\mcT$ may not exist,
but this only 
happens in the case where no branch model has an ordinal indexing an 
extender above $\delta(\mcT)$, which means that any such branch model is a last
model of the comparison. We call the strategy described above the uniqueness
strategy. When we say a premouse/extender 
model $M$ is iterable we mean $M$ is iterable via its uniqueness
strategy. Recall that a coiteration of $1$-small premice $M,N$ is a pair of
normal iteration trees $(\mcT,\mcU)$ according to the uniqueness strategies
where $\mcT$ is on $M$, $\mcU$ is on $N$ and the successor steps
$\alpha+1$ in the trees are determined by the extenders which constitute the
least disagreements between the $\alpha$-th models of $\mcT$ and $\mcU$. We say
that a coiteration $(\mcT,\mcU)$ is terminal if an only if the trees
$\mcT,\mcU$ have last models, and the last model of $\mcT$ is an initial
segment of the last model of $\mcU$ or vice versa. 

\begin{definition}\label{d0.1}
Let $\delta$ be a cardinal, or $\delta=\On$. Let $W$ be an iterable $1$-small
extender model. 

We say that $W$ is {\bf universal} for extender models of size less than
$\delta$ if and only if $W$ does not lose the coiteration with any iterable
$1$-small extender model $N$ of size less than $\delta$, that is: Let
$(\mcT,\mcU)$ be the pair of iteration trees according to the uniqueness
strategies corresponding to the terminal coiteration of $W$ with $N$, where
$\mcT$ is on  $W$ and $\mcU$ is on $N$. Then $\mcT,\mcU$ are of length less
than $\delta$, there is no truncation on the main branch of $\mcU$, and the
last model of $\mcU$ is an initial segment of the last model of $\mcT$. 
\end{definition}

We also fix the following notation. Given a regular cardinal $\delta$, we let 
\begin{equation}\label{i.e.sd}
S_\delta=\{x\in\ptm_\delta(\delta^+)\cap\mbfM
          \mid\cof^{\mbfM}(x\cap\delta)>\omega\}. 
\end{equation}
Here the cardinal successor $\delta^+$ is computed in $\mbfV$. 
Note that $\ptm_\delta(\delta^+)\cap\mbfM$ and $S_\delta$ are elements
of $\mbfM$. 
We then have the following variant of Lemma~\ref{lem:isstat}, whose the proof
is an almost literal repetition of the previous one.

\begin{lemma}\label{lem:isstat-gen}
Let $\delta>\omega_1$ be a regular cardinal and $\mbfM$ be a proper class
inner model such that $\delta^{+\mbfM}=\delta^+$. Then $S_\delta$ is
stationary. 
\end{lemma}

Using the terminology from the definition above we can state our
main theorem:

\begin{theorem}\label{t0.1}
Assume that $\mbfM$ is a proper class inner model, and that $\delta$ is a
regular cardinal in the sense of $\mbfV$. 
\begin{itemize}
\item[(a)] 
Granting that, in $\mbfV$, there is no proper class inner model with
a Woodin cardinal, $\delta>\omega_1$, and $S_\delta$ is
stationary, then the initial segment $\mbfK^{\mbfM}\cut\delta$ is universal
for all iterable premice in $\mbfV$ of cardinality less than $\delta$. 
\item[(b)] 
Granting that, in $\mbfM$, $0^\P$ does not exist,
and $\ptm_\delta(\delta^+)\cap\mbfM$ is stationary, then: If
$\delta>\omega_1$, then the initial segment $\mbfK^\mbfM\cut\delta$ is
universal for all iterable premice in $\mbfV$ of cardinality less than
$\delta$. Moreover, if $\delta=\omega_1$, then $\mbfK^{\mbfM}\cut\omega_2$ is
universal for all countable iterable premice in $\mbfV$.  
\end{itemize}
\end{theorem}

It was pointed out to us by the referee that our arguments can be run under
somewhat more relaxed assumptions with only minor amendments. The referee
suggested an improvement to Theorem~\ref{t0.1}, which we include in
Section~\ref{ss.referees-version}.  

The proof of Theorem~\ref{t0.1} is the heart of this paper and occupies
Section~\ref{sect:univ}. Given Lemma~\ref{lem:isstat-gen}, one gets the
obvious variant of Theorem~\ref{t0.1} where the assumption on stationarity of
$S_\delta$ is replaced with $\delta^{+\mbfM}=\delta^+$. Part (b) of the
theorem gives the same kind of 
conclusion as (a), but includes the important case $\delta=\omega_1$ which is
missing in (a), and uses a hypothesis on $\ptm_\delta(\delta^+)\cap\mbfM$
which is weaker than that in (a). The cost here is a more restrictive
anti-large cardinal hypothesis. We believe part (b) can be merged with part
(a) in the sense that the anti-large cardinal assumption in (b) can be
replaced with that in (a) (and part (a) can then be deleted), but at the
moment we do not see how to run the arguments in the proof of (b) under the
anti-large cardinal assumption in (a). 
On the other hand, the
conclusion in (b) that $\mbfK^{\mbfM}\cut\omega_2$ is universal for
countable mice in $\mbfV$ cannot be improved by replacing $\omega_2$ by
$\omega_1$. In fact, such conclusion can be false even if $\mbfM=\mbfV$, by a
result of Jensen \cite{spmwv}. For example, it is consistent (under mild large 
cardinal assumptions) that there is a countable 
mouse $m$ that iterates past $\mbfK\cut\omega_1.$ In fact, starting with
${\mbfK}={\mbfL}[E]$ such that the universe of $\mbfL[E]$ is the same as
that of $\mbfL[U]$ for some normal measure $U$ on some $\kappa$, Jensen
\cite{spmwv} describes  a forcing extension of ${\mbfK}$ that turns $\kappa$
into $\omega_1$, and adds a countable mouse $m$ such that
$(J_\beta^E,E_\beta)$ is an iterate of $m$; here $\beta$ is the index of the
extender coding the measure $U$. 
In Section \ref{sec:dis} we describe a scenario for a proof of the strengthening
of Theorem \ref{t0.1}(a) where the assumption on stationarity of $S_\delta$ 
is replaced with the assumption that $\ptm_\delta(\delta^+)\cap{\mbfM}$ is
stationary, and also discuss a variant of the theorem where the assumption on
non-existence of inner models with Woodin cardinals is weakened to
non-existence of such inner models from the point of view of~${\mbfM}$. 

\subsection{Applications}

\begin{corollary}\label{c0.1}
Assume that $\mbfM$ is a proper class inner model. 
\begin{itemize}
\item[(a)]
Granting that, in $\mbfV$, there
is no proper class inner model with a Woodin cardinal, and $S_\delta$
is stationary for proper class many cardinals $\delta$ that are regular in
$\mbfV$, then the model $\mbfK^{\mbfM}$ is universal for all set sized iterable
$1$-small premice in $\mbfV$.  
\item[(b)]
Under the assumptions that, in $\mbfV$, there is no proper class inner
model with a strong cardinal, and $\ptm_\delta(\delta^+)\cap\mbfM$ is
stationary for proper class many cardinals $\delta$ that are regular in
$\mbfV$, we conclude that $\mbfK^{\mbfM}$ is universal with respect to all
extender models in $\mbfV$ so, in particular, the models $\mbfK^{\mbfM}$ and
$\mbfK=\mbfK^\mbfV$ are Dodd-Jensen equivalent. 
\end{itemize}
\end{corollary} 

\begin{proof}
The conclusions in (a) are obvious consequences of
Theorem~\ref{t0.1}. The conclusion on universality of $\mbfK^{\mbfM}$
in (b) follows from the well-known fact that in the absence
of proper class inner models with strong cardinals, universality with
respect to set-sized extender models implies universality with respect to
proper class extender models (see, for example, Zeman \cite[Lemma
6.6.6]{imlc}). Since $\mbfK^{\mbfM}$ and $\mbfK$ are both universal in the
sense of $\mbfV$, they are Dodd-Jensen equivalent.  
\end{proof}

If there is no proper class inner model with a strong cardinal then the above
corollary is a generalization of the well-known fact that $\mbfK$ is generically
absolute; see Zeman~\cite{imlc}. In the case of generic absoluteness we let
$\mbfV$ play the role of 
$\mbfM[G]$ where $G$ is a generic filter for some poset in $\mbfM$, and
conclude that $\mbfK^{\mbfM}=\mbfK$. If $\mbfV$ is not a generic
extension of $\mbfM$ we of course cannot expect that the core model will
remain unchanged, but by the above corollary we still can expect that the
core models of $\mbfM$ and $\mbfV$ are close. 

A well-known example which
illustrates the situation where $\mbfV$ is not a generic extension of $\mbfM$
is the following: The model $\mbfM$ is obtained as an ultrapower of $\mbfV$
via some extender in $\mbfV$. Then $\mbfK^{\mbfM}$ is distinct from
$\mbfK$, yet, if for instance $0^{\P}$ does not exist, $\mbfK^{\mbfM}$ is a
normal iterate of $\mbfK$; see again Zeman~\cite{imlc}. Notice that the 
hypothesis of the above corollary is satisfied here, as $\delta$ can be taken
to be the cardinal successor of any singular strong limit cardinal of
sufficiently high cofinality, assuming there is no proper class inner model
with a Woodin cardinal. In this case $\delta^{+\mbfM}=\delta^+$ which
guarantees that $S_\delta$ is stationary in $\mbfV$ by
Lemma~\ref{lem:isstat-gen}. Thus, the situation just described is an instance
of Corollary~\ref{c0.1}. 

It is well-known that if there is an inner model with a strong cardinal then
the conclusion in Corollary \ref{c0.1}(b) is no
longer true, that is, it may happen that $\mbfK^{\mbfM}$ 
is strictly below $\mbfK$ in the Dodd-Jensen pre-well-ordering, yet $\mbfM$
computes cardinal successors correctly for proper class many regular cardinals
in $\mbfV$. This example is essentially described in Steel \cite[\S 3]{cmip}:
It suffices to let $\mbfV$ be  
the minimal iterable proper class extender model with one 
strong cardinal, say $\kappa$, and construct a linear iteration of length
$\On$ by iteratively applying sufficiently large extenders on the images of the
extender sequence of $\mbfV$. Then let $\mbfM$ be the initial segment of the
resulting direct limit of ordinal length. We described this example here
because it also shows that we cannot 
expect to prove universality of $\mbfK^{\mbfM}$ with respect to proper class
extender models in $\mbfV$ once $\mbfV$ sees proper class inner models with
strong cardinals. In the example just described, $\mbfV$ sees a proper class
inner model with a strong cardinal, namely itself, but $\mbfM$ does not see
any such inner model.  

Recall that, by Lemma~\ref{lem:isstat-gen}, the conclusions in 
Theorem~\ref{t0.1}(b) hold if the assumption of stationarity of $S_\delta$ is
replaced with the assumption that $\mbfM$ computes the cardinal successor of
$\delta$ correctly. For $\delta=\omega_1$, this assumption follows from 
the requirement that $\omega_2^{\mbfM}=\omega_2^{\mbfV}$. In contrast to this, 
notice that if $0^\sharp$ exists and $g$ is
$\mathsf{Coll}(\omega,<\omega_1^\mbfV)$-generic over $\mbfV$, then
$\mbfM=\mbfL[g]$ is an inner model of $\mbfV[g]$ which computes $\omega_1$
correctly, yet $0^\sharp\notin\mbfM$. There is also a finer 
folklore result that traces back at least to Hjorth's thesis (see the
Claim on page 430 of Hjorth \cite{Hj}) that not much can be concluded about
the mice of an inner model ${\mbfM}$ if merely $\omega_1^{\mbfM}=\omega_1$ is
assumed. In particular, the following holds:  

\begin{fact}
If $0^\sharp$ exists, then there exists an inner model ${\mbfM}$
which is a set forcing extension of ${\mbfL}$ and computes $\omega_1$
correctly. Thus $0^\sharp\notin {\mbfM}$.
\end{fact}

\begin{question}
Is there a class forcing extension of ${\mbfL}$ such that no inner model
with the same $\omega_1$ is a set forcing extension of ${\mbfL}$?
\end{question}

On the other hand, as hinted by Lemma \ref{lem:Fried}, the assumption
$\omega_2^{\mbfM}=\omega_2$ is significantly different. The following
corollary is another piece of evidence supporting this. The conclusions of the
corollary hold
under the assumption that $\ptm_\delta(\delta^+)\cap\mbfM$ is stationary for
some regular $\delta\ge\omega_1$, but we stated it for $\delta=\omega_1$ as we
find this instance finest. 

\begin{corollary} \label{cor:smallmice}
Assume that ${\mbfM}$ is a proper class inner model, and that $\ptm_{\omega_1}
(\omega_2)\cap{\mbfM}$ is stationary. Let $r$ be a sound mouse in ${\mbfV}$
projecting to $\omega$ such that either $r$ is below $0^\P$ or $r=0^\P$. Then
$r\in M$.  
\end{corollary}

\begin{proof}
This follows from Theorem \ref{t0.1}: Assume the conclusion fails, so
$\mathbf{K}^\mathbf{M}$ exists  
and does not reach $0^\P$. Then, by Theorem \ref{t0.1}(b),
$\mathbf{K}^\mathbf{M}||\omega_2$ is  
universal for countable mice in ${\mbfV}$, and it follows from a standard
comparison argument that such  
mice must be an initial segment of $\mathbf{K}^\mathbf{M}||\omega_2$.
\end{proof}

We reiterate that we made no anti-large cardinal assumptions on ${\mbfV}$ in
Corollary  \ref{cor:smallmice}. Naturally, our results relativize so, for
example, for any real $x\in{\mbfM}$, Corollary \ref{cor:smallmice} holds for
countable $x$-mice below $x^\P$. 
Finally we stress that, as already mentioned in the above corollary, the proof
of Theorem \ref{t0.1}(b) actually shows the following:  
\begin{quote}
If $0^\P$ exists in $\mbfV$, then it is actually in $\mbfM$. 
\end{quote}

\begin{corollary}\label{c0.2}
Assume that $\mbfM$ is a proper class inner model, and that $\delta$ is a
regular cardinal in the sense of $\mbfV$. Assume further that one of the
following holds: 
\begin{itemize}
\item[(a)] 
In the sense of $\mbfV$, there is no inner model with a Woodin
cardinal, $a^\dagger$ exists for every real $a$ in $\mbfV$, $\delta>\omega_1$,
and $S_\delta$ is stationary. 
\item[(b)] 
In $\mbfV$, $0^\P$ does not exist, $\mbbR^{\mbfV}$ is closed under sharps,
$\delta\ge\omega_1$, and $\ptm_\delta(\delta^+)\cap\mbfM$ is stationary.
\end{itemize}
Then $\mbfM$ is $\Sigma^1_3$-correct. 
\end{corollary}

\begin{proof}
Let $A\subseteq\mbbR$ be a nonempty $\Pi^1_2$-set. By Schindler \cite{spck}
under assumption (a) and by Steel-Welch \cite{s13a2ui} or Schindler \cite{sck}
under assumption (b), there is a premouse $M$ such that $A\cap M$ is
nonempty; pick an $r$ in this intersection. Without loss of generality we may
assume $M$ is sound and projects to $\omega$, hence is countable. By
Theorem~\ref{t0.1}, the model 
$\mbfK^{\mbfM}\cut\max(\delta,\omega_2)$ iterates past $M$. We conclude that
$r\in\mbfK^{\mbfM}\subseteq\mbfM$, so $A\cap\mbfM$ is nonempty. 
\end{proof}

Instead of quoting Schindler~\cite{spck} in the above proof we could
alternatively quote Theorem~7.9 in Steel~\cite{cmip}; which, applied inside the
model $r^\dagger$ for $r\in A$, yields a mouse $M$ as in the proof. 
By Theorem~\ref{t0.1}, in Corollary~\ref{c0.2}(b) we could equivalently assume
that $0^\P\notin{\mbfM}$, as such a sharp, if it exists in $\mbfV$, must
already be in $\mbfM$ (as pointed out above). Also, similarly as with
Theorem~\ref{t0.1}, appealing to Lemma~\ref{lem:isstat-gen}, there is an
obvious variant of the above corollary where the assumption of stationarity of
$S_\delta$, resp. $\ptm_\delta(\delta^+)\cap\mbfM$ is replaced with the
requirement that $\delta^{+\mbfM}=\delta^+$.  

\subsection{Acknowledgments}

The first author was supported in part by NSF Grant DMS-0801189. The second author was 
supported in part by NSF grant DMS-0500799, and by a UCI CORCL Special Research Grant. While 
preparing this paper, the first author visited UCI on two occasions in 2009 and 2010. The visits were 
sponsored by the second author's NSF grant, and by the Logic Seminar, Mathematics Department, UC 
Irvine.

Versions of these results were presented by the first author at BEST 2010, and by the second 
author at the 2010 ASL Annual meeting in Washington (partially supported by the UCI CORCL 
grant), in Singapore at a Summer meeting in 2011 (partially supported by the
John Templeton  foundation), and in M\"unster at the Second Conference on the
core model induction and hod mice (partially supported by the first author's
NSF grant). 

We want to thank Matt Foreman for bringing to our attention the results of Abraham \cite{fwch}, that 
suggested the relevance of the set we call $S_\delta$. We also want to thank a number of people for 
conversations related to this paper: Uri Abraham, Menachem Magidor, John Steel, Philip Welch, and 
Hugh Woodin. 

We would also like to thank the anonymous referee for their very thorough review
of our paper, which must have taken a lot of work, and for many thoughtful
suggestions which helped us improve the paper. The second author would also
like to thank the referee for pointing out the error in his \cite{imlc} on
page~274.  

The first author wants to thank the family of the second author for being such wonderful hosts during his 
visits to UCI.

\section{Iterability and background core model theory} \label{sect:cmt}

In this section we state facts from core model theory relevant to the proof
of our main theorem, Theorem~\ref{t0.1}, and establish absoluteness of
iterability for proper class extender models between two proper class inner
models. The proof is standard, and we give it to make this paper
self-contained. We also draw two useful corollaries, not needed in our main
argument.    

Let us make some remarks concerning notation. As already mentioned in the
introduction, we will be using Mitchell-Steel indexing of extenders in order
to stay close enough to our primary reference Steel \cite{cmip}. 
If $M$ is an extender model, $E^M$ will denote its extender sequence, and
$E^M_\beta$ will denote the extender on the sequence $E^M$ indexed by the
ordinal $\beta$; here $\beta$ is the cardinal successor of the natural length
of $E^M_\beta$ as calculated in $\ult(M,E^M_\beta)$. Recall that if $\mcT$ is
an iteration tree then $E^\mcT_\alpha$ is the 
extender coming from the $\alpha$-th model $M^\mcT_\alpha$ used to form an
ultrapower in $\mcT$. If $b$ is a branch through $\mcT$, we write $\pi^\mcT_b$
to denote the iteration map along $b$ if the map exists. If $i\in b$, we write
$\pi^\mcT_{i,b}$ to denote the iteration map along the portion of $b$ above
$i$ (including $i$) if the map exists. Thus $\pi^\mcT_b=\pi^\mcT_{0,b}$ and
$\pi^\mcT_{i,b}$ exists iff no truncations occur on $b$ above $i$. The
domain of $\pi^\mcT_{i,b}$ is either $M^\mcT_i$ or the result of the
corresponding truncation of $\mcM^\mcT_i$. 
Recall also that if $\mcT$ is a normal tree then $\delta(\mcT)$ is the
supremum of all ordinals  
$\rho(E^\mcT_\alpha)$; here $\rho(E^\mcT_\alpha)$ is the strict supremum of
the generators of $E^\mcT_\alpha$. Next, $E(\mcT)$ is the union of all
extender sequences $E^{M^\mcT_\alpha}\rst\rho(E^\mcT_\alpha)$ where
$\alpha<\lht(\mcT)$, and $M(\mcT)$ is the premouse of height $\delta(\mcT)$
whose extender sequence is $E(\mcT)$. If $\mcT$ is of limit length then
$\delta(\mcT)$ is the supremum of all lengths $\lht(E^\mcT_\alpha)$ of
extenders used in $\mcT$, or equivalently, all iteration indices in $\mcT$.

Recall also that if $\mcT$ is an iteration tree of limit length and $b,c$
are two distinct cofinal well-founded branches through $\mcT$, then
$\delta(\mcT)$ is a Woodin cardinal in $M^{\mcT}_b\cap M^{\mcT}_c$ where
$M^{\mcT}_b,M^{\mcT}_c$ are the models at the end of the corresponding
branches; see for instance Mitchell-Steel \cite{fsit}. A similar fact is true
even in the case where the branch models $\mcM^\mcT_b,\mcM^\mcT_c$ are not
well-founded, but $\delta(\mcT)\in\wfp(M^\mcT_b)\cap\wfp(M^\mcT_c)$ where
$\wfp(\mcM^\mcT_b)$ denotes the transitivised maximal well-founded initial
segment of $\mcM^\mcT_b$, and similarly for $\mcM^\mcT_c$. In this case,
$\delta(\mcT)$ has the Woodinness property with respect to all subsets of
$\delta(\mcT)$ which are in $\mcM^\mcT_b\cap\mcM^\mcT_c$. Recall from the
introduction that, granting there is no
proper class inner model with a Woodin cardinal, any iterable proper class
extender model has precisely one iteration strategy called the {\bf
uniqueness} strategy; this is the strategy that picks the unique  
cofinal well-founded branch at each limit step of the iteration.  

The following lemma is a slight variation of Steel \cite[Lemma 5.12]{cmip},
and may 
be viewed as a generalization thereof, as we allow $\mbfV$ be any outer
extension of $\mbfM$, and not merely a generic extension. The proof thus
requires an additional appeal to Shoenfield absoluteness, which somewhat
changes not only the structure of the proof but the statemtent of the
result itself. We give the proof below; one notable difference is that the proof of Steel 
\cite[Lemma 5.12]{cmip} uses, in an important
way, the assumption that every set in the ground model has a sharp. The
existence of sharps is, of course, an immediate 
consequence of the initial setting where $\Omega$ is a measurable cardinal and
our universe of interest is the rank-initial segment of height $\Omega$. However,
this assumption may be false if we want to run the proof of Theorem~\ref{t0.1}
in $\zfc$ alone. There is
one more difference between the two proofs which makes our lemma in a sense
``less general" than Steel \cite[Lemma 5.12]{cmip}, as we are making the 
assumption that there is no inner model with a Woodin cardinal in the sense of
$\mbfV$. The proof of Steel \cite[Lemma 5.12]{cmip} only requires that no Woodin
cardinals exist in the extender model we want to iterate. We do not know if the 
asumption on the non-existence of an inner model with a Woodin cardinal in the
sense of $\mbfV$ can be weakended by replacing $\mbfV$ with the model we
intend to iterate or with our inner model $\mbfM$, even in the case where we
allow our universe to be closed under sharps. 

\begin{lemma}\label{cmt:l1}  
Assume that $\mbfM$ is a proper class inner model, and that in $\mbfV$ there is no
proper class inner model with a Woodin cardinal. In $\mbfM$, let $R$ be a
proper class extender model. The following are equivalent:
\begin{itemize}
\item[(a)] 
In $\mbfM$, the model $R$ is normally iterable via the uniqueness strategy. 
\item[(b)] 
In $\mbfV$, the model $R$ is normally iterable via the uniqueness strategy. 
\end{itemize}
Here ``iterable" means iterable with respect to trees of set-sized length. 
\end{lemma}

\begin{proof} 
We give a proof of the implication (a) $\Rightarrow$ (b). The converse is proved
in a similar way, and the argument is much simpler. Toward a contradiction,
assume that in $\mbfV$ there is a normal iteration tree $\mcT$ on $R$
witnessing that (a) $\Rightarrow$ (b) is false. We focus on the essential case
where the length 
of $\mcT$ is a limit ordinal. Thus, in $\mbfV$ the tree $\mcT$ does not have a
unique cofinal well-founded branch. Using a reflection argument, we can find a
regular cardinal $\alpha>\delta(\mcT)^+$ such that $\mcT$, when viewed as a
tree on $R'=R\cut\alpha$, does not have a unique cofinal well-founded branch
in $\mbfV$. Notice that the transitive closure of $\mcT$ is then of size
$\alpha$. Let $\theta>\alpha$ be a limit cardinal with large cofinality, and
let $T$ be the lightface theory of $H^{\mbfV}_\theta$ in the language of $\zfc$. 
Let $\varphi(x,y,z)$ be the conjunction of the following statements:
\begin{itemize}
\item $x$ is a transitive model of $T$.  
\item $z,y\in x$, and $z$ is a regular cardinal in $x$. 
\item There is a normal iteration tree $\mcU\in x$ on the premouse $y$ of
      limit length such 
      that $\card(\trcl(\mcU))^x=z$, $\delta(\mcU)^{+x}<z$, and in $x$, the
      tree $\mcU$ does not have a unique cofinal well-founded branch. 
\item There is no $a\in x$ such that $a$ is a bounded subset of $z$ and
      $J_z[a]$ is a model with a Woodin cardinal.  
\end{itemize}
Here $\trcl(v)$ denotes the transitive closure of $v$. 

Let $\beta\ge\card(H_\theta)$, and let $G$ be a filter generic for
$\col(\omega,\beta)$ over $\mbfV$. Immediately from the properties of $\mcT$
and from our assumption that no inner model with a Woodin cardinal exists in
the sense of $\mbfV$ we conclude that $\varphi(H^{\mbfV}_\theta,R',\alpha)$
holds in $\mbfV[G]$. As $H^{\mbfV}_\theta$ 
is countable in $\mbfV[G]$ and $R',\alpha$ are countable in $\mbfM[G]$, it
follows by Shoenfield absoluteness that in $\mbfM[G]$ there is some $H$ such
that $\mbfM[G]\models\varphi(H,R',\alpha)$. 

Pick a $\col(\omega,\beta)$-name $\dot{H}\in\mbfM$ and some regular $\theta'$
such that $H=\dot{H}_G$ and $\beta,\dot{H},R'\in H^{\mbfM}_{\theta'}$. Finally,
pick a condition $p\in\col(\omega,\beta)$ such that $p$ forces
$\varphi(\dot{H},\check{R}',\check{\alpha})$ over $H^{\mbfM}_{\theta'}$. 
From now on, work in $\mbfM$. Let $X$ be a
countable elementary substructure of $H^{\mbfM}_{\theta'}$ such that
$\dot{H},R',\beta,p\in X$, and let 
 $$ \sigma:\tilde{H}\to H^{\mbfM}_{\theta'} $$
be the 
inverse of the Mostowski collapsing isomorphism that arises from collapsing
$X$. Let $\dot{\bar{H}},\bar{R},\bar{\alpha},\bar{\beta},\bar{p}$ be the
preimages of $\dot{H},R',\alpha,\beta,p$ under $\sigma$. Pick a filter
$g$ extending $\bar{p}$ that is generic for $\col(\omega,\bar{\beta})$ over
$\tilde{H}$, and let $\bar{H}=\dot{\bar{H}}_g$. Then $\bar{\alpha}$ is a
regular cardinal in $\bar{H}$, and there is a normal iteration tree
$\bar{\mcU}\in\bar{H}$ on $\bar{R}$ of limit length such that
$\card(\trcl(\bar{\mcU}))^{\bar{H}}=\bar{\alpha}$, 
$\delta(\bar{\mcU})^{+\bar{H}}<\bar{\alpha}$ and, in $\bar{H}$, the tree
${\bar{\mcU}}$ does not have a unique cofinal well-founded branch.   

We next observe that for every limit ordinal $\lambda^*<\bar{\lambda}$, the
tree $\bar{\mcU}$ has at most one cofinal well-founded branch in $\mbfM$: If
$b\neq c$ were two distinct cofinal well-founded branches through
$\bar{\mcU}\rst\lambda^*$, then $\delta(\bar{\mcU}\rst\lambda^*)$ would be
Woodin in $M^{\bar{\mcU}}_b\cap M^{\bar{\mcU}}_c$. If
$\nu>\delta(\bar{\mcU}\rst\lambda^*)$ is 
least among all ordinals indexing extenders in $M^{\bar{\mcU}}_b$ or
$M^{\bar{\mcU}}_c$ (note such a $\nu$ always exists as
$\lambda^*<\bar{\lambda}$), assume without loss of generality that $\nu$
indexes an extender in $M^{\bar{\mcU}}_b$. Then
$\delta(\bar{\mcU}\rst\lambda^*)$ is Woodin 
in $M^{\bar{\mcU}}_b\cut\nu$ and, using the elementarity of the iteration maps,
we conclude that there are $\delta'<\nu'$ such that $\nu'$ indexes an extender
in $\bar{R}$, and $\lambda'$ is Woodin in $\bar{R}\cut\nu'$. Since
$\bar{R}\in\tilde{H}$, we may apply $\sigma$ to $\bar{R}$ and conclude that
$\sigma(\nu')$ indexes an extender in $R'$ and $\sigma(\delta')$ is Woodin in
$R'\cut\sigma(\nu')=R\cut\sigma(\nu')$. Since we are assuming $R$ is iterable
in the 
sense of $\mbfM$, iterating the extender on the $R$-sequence with index $\nu$
through the ordinals will produce a proper class inner model with a Woodin
cardinal, a contradiction.  

The same argument for $\lambda^*=\bar{\lambda}$ shows that in $\mbfM$, if
$b\neq c$ are two cofinal well-founded branches through $\bar{\mcU}$ then no
$\nu>\delta(\bar{\mcU})$ indexes an extender in $M^{\bar{\mcU}}_b$ or
$M^{\bar{\mcU}}_c$. Now at most one of the branches $b,c$ involves a
truncation; otherwise the standard argument can be used to produce two
compatible extenders applied in $\bar{\mcU}$, and thereby obtain a
contradiction. If one of 
the branches involves a truncation, let it be without loss of generality $c$;
then $M^{\bar{\mcU}}_b\unlhd M^{\bar{\mcU}}_c$. If neither of the branches
involves a truncation, let $b$ be such that 
$M^{\bar{\mcU}}_b\unlhd M^{\bar{\mcU}}_c$. In either case, $\bar{\lambda}$ is
Woodin in $M^{\bar{\mcU}}_b$ and there is an iteration map 
$\pi_b:\bar{R}\to M^{\bar{\mcU}}_b$. By the elementarity of $\pi_b$, there is
a Woodin cardinal in $\bar{R}$. And by the elementarity of $\sigma$, there is
a Woodin cardinal in $R'$. Since $R'=R\cut\alpha$ and $\alpha$ is a cardinal,
$R$ is a proper class model with a Woodin cardinal, a contradiction. In
particular, this shows that $\bar{\mcU}$ has no cofinal well-founded branch in
$\bar{H}$. (Because our initial hypothesis stipulates that such a branch
cannot be unique.)

Since $\sigma\rst\bar{R}:\bar{R}\to R'$ is elementary and $\bar{\mcU}$ picks
unique cofinal well-founded branches, we can copy $\bar{\mcU}$ onto an
$R'$-based iteration tree $\mcU'$ via $\sigma\rst\bar{R}$. Let $b$ be a
cofinal well-founded branch through $\mcU'$; such a branch exists by our
assumption on the iterability of $R$ in $\mbfM$. By the properties of the
copying construction, $b$ is also a cofinal well-founded branch through
$\bar{\mcU}$. Now assume $\gamma\ge\bar{\alpha}$ is an ordinal in $\bar{H}$,
and $h\in\mbfM$ is a filter on $\col(\omega,\gamma)$ that is generic over
$\bar{H}$. By our initial setting,
$\card(\trcl(\mcU))^{\bar{H}}=\bar{\alpha}\le\gamma$; hence both $\gamma$ and
the transitive closure of $\bar{\mcU}$ are countable in $\bar{H}[h]$. Since    
the statement ``There is a cofinal branch $x$ through $\bar{\mcU}$ such that
the ordinals of $M^{\bar{\mcU}}_x$ have an initial segment isomorphic to
$\gamma$" is $\Sigma^1_1$ in the codes, if $\bar{\mcU}$ has such a branch, then
some such branch $c$ exists in
$\bar{H}[h]$. If $M^{\bar{\mcU}}_c$ is well-founded, then by the previous
paragraph, $c$ is unique, and by the homogeneity of $\col(\omega,\gamma)$, such
a branch must be in $\bar{H}$, which contradicts the last conclusion of the
previous paragraph. Hence $M^{\bar{\mcU}}_c$ must be ill-founded. If
$\On\cap M^{\bar{\mcU}}_b<\On\cap\bar{H}$ put $\gamma=\On\cap M^{\mcU}_b$, and
apply the above argument to the $\Sigma^1_1$-statement ``There is a cofinal
branch $x$ through $\bar{\mcU}$ such that the ordinals of $M^{\bar{\mcU}}_x$
are isomorphic to $\gamma$". We obtain a cofinal well-founded branch through
$\bar{\mcU}$ in $\bar{H}[h]$, a contradiction. (Here we are using the fact
that there are unboundedly many cardinals in $\bar{H}$.) It follows that
$M^{\bar{\mcU}}_b$ has height at least that of $\bar{H}$. 

Let $\tau=\delta({\bar{\mcU}})^{+\mbfL[E(\bar{\mcU})]}$. Set
$\gamma=\alpha$, and let $h\in\mbfM$ be generic for $\col(\omega,\gamma)$ over
$\bar{H}$. The arguments from the previous paragraph imply that there is a
branch $c\in\bar{H}[h]$ that has an initial segment isomorphic to $\gamma$. As
follows from the previous paragraph, such a branch is ill-founded. Let
$\gamma'$ be a cardinal in $\bar{H}$ larger than the ordinal of the
well-founded part of $M^{\bar{\mcU}}_c$, and let $h'$ be a filter generic for
$\col(\omega,\gamma')$ over $\bar{H}$. As $\On\cap M^{\bar{\mcU}}_b>\gamma'$,
as before we conclude that there is a branch $c'$ in $\bar{H}[h']$ with an
initial segment isomorphic to $\gamma'$. (Here we again make use of the fact
that cardinals of $\bar{H}$ are unbounded in $\bar{H}$.) 

Since the well-founded part of $M^{\bar{\mcU}}_{c'}$ is strictly longer than the well-founded part of
$M^{\bar{\mcU}}_c$, we conclude that $c\neq c'$. And since both
$\gamma,\gamma'$ are larger than $\alpha>\tau$, the ordinal
$\delta(\bar{\mcU})$ is a Woodin cardinal in $J_\alpha[E(\bar{\mcU})]$. Since
$\bar{\mcU}\in\bar{H}$, so is $J_\alpha[E(\bar{\mcU})]$, but this contradics
the last clause in $\varphi$ which holds in $\bar{H}$ by construction. 
\end{proof}

As a by-product, we note the following interesting consequence of Lemma~\ref{cmt:l1}.  

\begin{corollary}\label{cmt:c1}
Assume that in $\mbfV$, there is no proper class inner model with a Woodin
cardinal. Let $R,R'$ be proper class extender models and $\sigma:R\to R'$ be
an elementary map. Then $R$ is normally iterable if and only if $R'$ is
normally iterable. Here ``normally iterable" means that there is a uniqueness
normal iteration strategy defined on all set-sized normal iteration trees. 
\end{corollary}

\begin{proof}
Assume $R$ is iterable in $\mbfV$. Applying Lemma~\ref{cmt:l1} to the
situation where $\mbfM=R$, we conclude that $R$ is internally normally
iterable, that is, $R$ satisfies the statement ``$R$ is normally iterable". By
the elementarity of the map $\sigma$, the model $R'$ is internally normally
iterable. Applying Lemma~\ref{cmt:l1} again, this time with $\mbfM=R'$, we
conclude that $R'$ is normally iterable in the sense of $\mbfV$. 
\end{proof} 

Recall that an extender model $R$ is {\bf fully iterable} if and only if
$R$ is iterable with respect to linear compositions of arbitrary (hence even
class-sized) normal iteration trees. 
When working in a universe below the measurable cardinal $\Omega$, normal/full
iterability via the uniqueness strategy with respect to set-sized trees
automatically guarantees normal/full iterability via the uniqueness strategy
with respect to class-sized trees.  

Working in the universe below $\Omega$, fix a stationary class of
ordinals that will serve as a reference class for the notion of {\em
thickness}. In 
our applications we will assume for simplicity that this class consists of
inaccessibles, so we let this class be the class $A_0$ as defined at
the end of Steel \cite[\S 1]{cmip}. Then $\Omega$ is $A_0$-thick in $\Kc$. 

\begin{corollary}\label{cmt:c2}
Assume that in $\mbfV$, there is no proper class inner model with a Woodin
cardinal. Let $R$ be a proper class extender model that is fully iterable and
such that $\Omega$ is $A_0$-thick in $R$, and let $F$ be an $R$-extender such
that $\ult(R,F)$ is well-founded. Denote the transitive collapse of
$\ult(R,F)$ by $R'$. Then $R'$ is fully iterable. 
\end{corollary}

\begin{proof}
By Corollary~\ref{cmt:c1} and the discussion immediately following, $R'$ is
normally iterable with respect to class-sized normal iteration trees. It
follows that there is an extender model $R^*$ that is a common normal iterate
of both $R$ and $R'$. Since $\Omega$ is $A_0$-thick in $R$, the corresponding
normal iteration tree on $R'$ does not involve any truncation on its main
branch, and there is an iteration map $\pi:R'\to R^*$ that is fully
elementary. Now $R^*$ is fully iterable, as it is an iterate of $R$. The
existence of $\pi$ then guarantees that $R'$ is fully iterable as well. 
\end{proof} 
  
Finally we will need a well-known lemma on the absorption of extenders by the
core model; this criterion is similar to Theorem 8.6 in \cite{cmip}. Although
the proof in \cite{cmip} is inductive, our assumption allow us to avoid this
kind of induction. Additionally, there is no need to impose any variant of the
initial segment condition on the extender we intend to absorb by $\mbfK$. 

\begin{lemma}\label{cmt:abs}
Assume there is no proper class model with a Woodin cardinal. Let $F$ be a
$(\kappa,\lambda)$-extender over $\mbfK$ where $\lambda$ is a
$\mbfK$-cardinal, and let $W$ be a proper class extender 
model witnessing the $A_0$-soundness of $\mbfK\cut\delta$ where $\delta$ is a
cardinal larger than $\lambda$. Assume further that the following hold.
\begin{itemize}
\item[(a)] The phalanx $\langle W,W',\lambda\rangle$ is normally iterable 
           where $W'=\ult(W,F)$.
\item[(b)] $E^W\rst\nu=E^{W'}\rst\nu$ where $\nu=\lambda^{+W'}$. 
\end{itemize} 
Then $F\in W$. 
\end{lemma}

\begin{proof}
We give a sketch of the proof. Compare $W$ against the phalanx 
$\langle W,W',\lambda\rangle$. Let $\mcT,\mcU$ be the iteration
trees arising in the comparison where $\mcT$ is on $W$ and $\mcU$ is on the
phalanx. Using the fact that $\Omega$ is $A_0$-thick in $W$ we conclude 
that $\mcT,\mcU$ have a common last model $W^*$, the model $W'$ is on the main
branch of $\mcU$ and the iteration maps $i^\mcT:W\to W^*$ and $i^\mcU:W'\to W^*$
exist. The critical point of $i^\mcU$ is $\ge\lambda$. Using the fact that $W$
witnesses the $A_0$-soundness of $\mbfK\cut\delta$ we conclude that
$\crp(i^\mcT)=\kappa$ and 
\begin{equation}\label{e.itiu-kappa}
i^\mcT(x)=i_\mcU\circ i_F(x)\mbox{ whenever }x\in\ptm(\kappa)\cap W
\end{equation}
where $i_F$ is the ultrapower map associated with $\ult(W,F)$. 
By assumption (b) the first extender $E^\mcT_\alpha$ used on the main branch
of $\mcT$ has length at least $\nu$, and since $\lambda$ is a $W$-cardinal,
the natural length of $E^\mcT_\alpha$ is at least $\lambda$. (Notice that here
the fact that $\lambda$ is a $W$-cardinal is only used to cover the case where
$\alpha=0$.) But then it follows from (\ref{e.itiu-kappa}) combined with the
fact that $\crp(i^\mcT_{\alpha+1,\infty}),\crp(i^\mcU)\ge\lambda$ that $F$ is
an initial 
segment of $E^\mcT_\alpha$. By applying the initial segment condition for
$E^\mcT_\alpha$ along with the fact that $\lht(E^\mcT_0)$ is a cardinal in
$\mcM^\mcT_\alpha$ if $\alpha>0$ we conclude $F\in W$.  
\end{proof}

In general, if $F$ is a $(\kappa,\lambda)$ extender that satisfies condition
(b) above, we say that $F$ {\bf coheres} to $W$.

\section{Universality} \label{sect:univ}

\subsection{Outline}

Throughout this section we will work with a fixed proper class inner model
$\mbfM$ and a cardinal $\delta$ which is regular and uncountable in $\mbfV$.  
In order to guarantee that the core model theory
is applicable, we will additionally assume that in 
$\mbfV$ there is no proper class inner model with a Woodin cardinal.  

Our argument is based on a combination of ideas coming from the proof of
universality of $\mbfK^c$ described in Jensen \cite{anfs} and, in the simpler
form reformulated for measures of order zero, also in Zeman \cite[\S 6.4]{imlc},
further from the proof of universality of $\mbfK^c$ in Mitchell-Schindler
\cite{uemwlc}, and finally from the Mitchell-Schimmerling-Steel proof of
universality of $\mbfK\cut\delta$, see  Schimmerling-Steel \cite{mcm}. But of
course, since we work under significantly restricted circumstances, we need to
do some amount of extra work.   

As it is usual in arguments of this kind, we will replace $\mbfK^{\mbfM}$, the
true core model in $\mbfM$, with a proper class fully iterable extender model
$W$ that in 
$\mbfM$ witnesses the $A_0$-soundness of $\mbfK^{\mbfM}\cut\delta'$ for a
sufficiently large $\delta'>\delta$. This means that,
working in $\mbfM$, the model $W$ extends $\mbfK^{\mbfM}\cut\delta'$, is
$A_0$-thick, and $\delta'\subseteq H^W_1(\Gamma)$ whenever $\Gamma$ is
$A_0$-thick in $W$; here $H^W_1(\Gamma)$ denotes the $\Sigma_1$-hull of
$\Gamma$ in $W$. (See Steel \cite[\S 3]{cmip} for details.) In our case,
taking $\delta'=\delta^+$ will suffice. We recall that
we are using Mitchell-Steel indexing of extenders, as this setting is
convenient when referring to Steel \cite{cmip}.  

Our general strategy is the following:
\begin{itemize}
\item[(A)] Assuming that in $\mbfV$, there is an iterable premouse of size
less than $\delta$  
           that iterates past $W\cut\delta$, we construct a $W$-extender with
	   critical point $\kappa$ and support $\lambda$ such that
	   $F\in\mbfM$ and coheres to $W$. (This does not require that 
	   $\ult(W,F)$ is well-founded.)
\item[(B)] We prove that $\ult(W,F)$ is well-founded and, letting $W'$ be its
           transitive collapse, prove that $\lambda^{+W}=\lambda^{+W'}$ and the
	   phalanx $(W,W',\lambda)$ is normally iterable in the sense of
	   $\mbfM$. We then apply Lemma~\ref{cmt:abs} inside $\mbfM$ to
	   conclude $F\in W$. This will yield a contradiction, as $F$ codes a
	   function which collapses $\lambda^{+W}$.   
\end{itemize}

\subsection{Part (A)}

We begin with (A). The construction in (A) is the same for both parts of
Theorem~\ref{t0.1}; as mentioned above, we will use the Mitchell-Schindler
approach to the proof of universality. Assume that 
\begin{equation}\label{u.e.s}
S\subseteq\ptm_\delta(\delta^+)\cap\mbfM
\end{equation}
is stationary in the sense of $\mbfV$. 

By Lemma~\ref{cmt:l1}, the model $W$ is normally iterable via the uniqueness
strategy in the sense of $\mbfV$, so in particular it is coiterable with any
iterable premouse in $\mbfV$. Assume $N$ is an iterable premouse in $\mbfV$ of
size less than $\delta$ with an iteration strategy $\Sigma$ that iterates past
$\mbfK^{\mbfM}\cut\delta$. This means that there is a pair $(\mcT,\mcU)$ of
normal iteration trees arising in the coiteration of $W$ against $N$, where
$\mcT$ is on $W$ and $\mcU$ is on $N$, $\mcT$ is according to the 
uniqueness strategy for $W$, $\mcU$ is according to $\Sigma$ and the 
length of $\mcU$ is $\delta+1$. We view $\mcT,\mcU$ as unpadded trees, so the
length of $\mcT$ is $\le\delta+1$. Also, we view $\mcT$ as an iteration tree on
$W\cut\delta^+$. 

Fix the following notation:
\begin{itemize}
\item $W'$ is the last model of $\mcT$ and $N'$ is the last model of $\mcU$.   
\item The notation $\kappa^{\mcT}_i,\pi^{\mcT}_{ij}$ is
      self-explanatory, and $\nu^{\mcT}_i$ is the iteration index associated
      with the $\nu_i$-th extender on $\mcT$, that is,
      $\nu^\mcT_i=\lht(E^\mcT_i)$. We of course fix the corresponding
      notation for $\mcU$. 
\item $W_i=M^{\mcT}_i$ and $N_i=M^{\mcU}_i$. 
\item $T$, resp. $U$, is the tree structure associated with the iteration tree 
      $\mcT$, resp. $\mcU$, and $<_T$, $<_U$ are the respective 
      tree orderings.
\item $b^{\mcT}$ is the main branch of $\mcT$ and $b^{\mcU}$ is the main
      branch of $\mcU$. Note that $\delta\in b^\mcU$.   
\item For any $i<\delta$ we write $\xi^{\mcT}(i)$, resp. $\xi^{\mcU}(i)$,
      to denote the immediate $<_T$-, resp. $<_U$-, predecessor of $i+1$. 
\end{itemize}  

\begin{lemma}\label{u.l.term}  
Under the above setting, the following hold:
\begin{itemize}
\item[(a)] There is no truncation point on $b^\mcT$. 
\item[(b)] For every $\alpha<\delta$ there is $i+1\in b^\mcT$ such that 
           $\pi^\mcT_{0,\xi(i)}(\alpha)<\kappa^\mcT_i=
            \crp(\pi^\mcT_{\xi(i),b^\mcT})$. Equivalently,  
	   $\pi^\mcT_{b^\mcT}(\delta)=\delta$. 
\end{itemize}
\end{lemma}

\begin{proof}
Both clauses follow by standard arguments similar to those used in the proof
of the termination of the comparison process and in the proof of universality. 
\end{proof}

Using a pressing down argument, we obtain a club
$C^*\subseteq b^\mcU$ which is a thread through the set of critical
points. Precisely, $\pi^\mcU_{\xi,\xi'}(\kappa^\mcU_i)=\kappa^\mcU_j$ whenever 
$\xi<\xi'$ are elements of $C^*$ such that $\xi=\xi^\mcU(i)$ and
$\xi'=\xi^\mcU(j)$. Of course $\min(C^*)\ge\xi^\mcU(i^\mcU)$ where
$i^\mcU+1$ is the last truncation point on $b^\mcU$. We will additionally
assume that elements of $C^*$ satisfy the following:

\begin{itemize}
\item[(i)] $\xi=\kappa^\mcU_i$ whenever $\xi\in C^*$ and $i$ is such that
               $\xi=\xi^\mcU(i)$,  hence $\xi=\crp(\pi^\mcU_{\xi,\xi'})$
	       for all $\xi'\in b^\mcU-(\xi+1)$. 
\item[(ii)] $\pi^\mcT_{0,\xi}(\xi)=\xi$.
\item[(iii)] $\xi$ is a limit point of the critical points on $b^\mcT$ if
             $\lht(\mcT)=\delta+1$, and $\xi>\lht(E^\mcT_\alpha)$ for all
	     $\alpha<\lht(\mcT)$ if $\lht(\mcT)<\delta$.  
\end{itemize}

That these requirements may be imposed without loss of generality follows from
the fact that they are true on a club. For (i) this is
obvious,  and for (ii) and (iii) we appeal to the proof of
Lemma~\ref{u.l.term}. Notice 
also that each element of $C^*$ is an inaccessible cardinal in $N'$, and since
$N'$ agrees with $W'$ below $\delta$, also in $W'$. 

\begin{lemma}\label{l4}
Given any $\xi\in C^*$, let $\tau_\xi=\xi^{+W}$, and $\tau'_\xi=\xi^{+W_\xi}$
if $\lht(\mcT)=\delta+1$ and $\tau'_\xi=\xi^{+W'}$ if $\lht(\mcT)<\delta$. 
Then for every $\xi\in C^*$ the following is true:  
\begin{itemize}
\item[(a)] $\xi$ is inaccessible in $W_i$ for all $i\in b^\mcT$. 
\item[(b)] $\tau'_\xi=\xi^{+N_\xi}=\xi^{+N'}$. 
\end{itemize}
\end{lemma} 

\begin{proof}
We do the proof for the case where $\lht(\mcT)=\delta+1$. If
$\lht(\mcT)<\delta$ the proof is similar but easier. 
Let $\xi\in C^*$ and $i$ be such that $\xi=\xi^\mcT(i)$. Since the critical
points on $b^\mcT$ are cofinal in $\xi$, necessarily $\kappa^\mcT_i\ge\xi$. The
critical point of the map $\pi^\mcT_{i+1,\delta}$ is above $\xi$, and $\xi$ is
inaccessible in $W'$ by the discussion immediately above the statement of the
lemma, so $\xi$ is an
inaccessible cardinal in $W_{i+1}$. Since $W_{i+1}=\ult(W_\xi,E^\mcT_i)$ and
$E^\mcT_i$ is an extender with critical point at least $\xi$, the ordinal $\xi$ 
is inaccessible in $W_\xi$. Since $\pi^\mcT_{0,\xi}(\xi)=\xi$, the ordinal
$\xi$ is inaccessible in all models $W_i$ for $i<_T\xi$. This proves
(a). Regarding (b), we have 
$\ptm(\xi)\cap W_{\xi^\mcT(i)}=\ptm(\xi)\cap W_{i+1}$, by our choice of
$\xi$ and general properties of iterations trees. Since
$\crp(\pi^\mcT_{i+1,\delta})>\xi$, we conclude  
$\tau'_\xi=\xi^{+W'}=\xi^{+N'}=\xi^{+N_\xi}$, where the last equality follows by
the same considerations for $\mcU$ as have been just done for 
$\mcT$. 
\end{proof}

The initial setting recorded above is more or less the same as in Jensen's
proof of universality of $\mbfK^c$ in Jensen \cite{anfs}. Although the next two
lemmata follow the general scenario of his proof, the arguments we use are significantly more local, 
and we need to do this for two reasons. First, we work in $\mbfV$ but on the other 
hand, the model whose universality we want to establish is in $\mbfM$, which
restricts our freedom in choosing the structures to work with. Second, we 
work below a regular cardinal $\delta$ given in advance which -- unlike the
situation in Jensen's proof -- we would like to be as small as possible, and
therefore it need not have any reasonable closure
properties. We stress that we are not making any assumptions on cardinal
arithmetic. These two lemmata comprise one of the key observations in
our argument.  

\begin{lemma}\label{u.l.pullback-1} 
Let $\xi\in\lim(C^*)$ and $\xi<\zeta<\tau_\xi$. Then there is a
$\xi^*\in C^*$ such that $\xi^*<\xi$ and, for all $\bar{\xi}\in C^*$ with 
$\xi^*\le\bar{\xi}<\xi$, we have 
 $$ (\pi^\mcU_{\bar{\xi},\xi})^{-1}\circ\pi^\mcT_{0,\xi}(\zeta)
  \in\rng(\pi^\mcT_{0,\bar{\xi}}). $$ 
\end{lemma}

\begin{proof}
Since $\xi<\zeta<\tau_\xi$, there is a set $a\in W$ such that 
$a\subseteq\xi\times\xi$ and $a$ is a well-ordering of $\xi$ of order-type 
$\zeta$. Let $\xi^*$ be the least element $\xi'$ of $C^*$ such that 
$\pi^\mcT_{0,\xi}(a)\in\rng(\pi^\mcU_{\xi',\xi})$. Then for
$\bar{\xi}\in[\xi^*,\xi)\cap C^*$ we have 
\[
\pi^\mcT_{0,\xi}(a)=\pi^\mcU_{\bar{\xi},\xi}(\pi^\mcT_{0,\xi}(a)\cap\bar{\xi})=
 \pi^\mcU_{\bar{\xi},\xi}\circ\pi^\mcT_{0,\bar{\xi}}(a\cap\bar{\xi}).
\] 
Here the first equality follows from the fact that 
$\bar{\xi}=\crp(\pi^\mcU_{\bar{\xi},\xi})$ and
$\pi^\mcU_{\bar{\xi},\xi}(\bar{\xi})=\xi$. To see the second equality notice
that  $\crp(\pi^\mcT_{\bar{\xi},\xi})\ge\bar{\xi}$, so 
$\pi^\mcT_{0,\xi}(a)\cap\xi=\pi^\mcT_{0,\bar{\xi}}(a)$. 
By the elementarity of all relevant maps, $a\cap\bar{\xi}$ is a well-ordering
of $\bar{\xi}$. Letting $\bar{\zeta}$ be the order-type of $a\cap\bar{\xi}$,
we conclude that $(\pi^\mcU_{\bar{\xi},\xi})^{-1}\circ\pi^\mcT_{0,\xi}(\zeta)
  =\pi^\mcT_{0,\bar{\xi}}(\bar{\zeta})$. 
  \end{proof}   

Given a map $f$ and a set $x$, we will write $f[x]$ to denote the pointwise
image of $x$ under $f$. 

\begin{lemma}\label{u.l.pullback-2} 
There is a $\xi^*\in C^*$ such that for every $\xi,\xi'\in C^*$ with 
$\xi^*\le\xi<\xi'$ we have 
$\pi^\mcU_{\xi,\xi'}\circ\pi^\mcT_{0,\xi}[W\cut\tau_\xi]\subseteq
 \pi^\mcT_{0,\xi'}[W\cut\tau_{\xi'}]$. Moreover, the set on the left is a
cofinal subset of the set on the right. 
\end{lemma}

\begin{proof}
To prove the inclusion, it suffices to show that 
$\pi^\mcU_{\xi,\xi'}\circ\pi^\mcT_{0,\xi}[\tau_\xi]\subseteq
 \pi^\mcT_{0,\xi'}[\tau_{\xi'}]$ for all $\xi<\xi'$ on a tail-end of $C^*$. 
The first step toward the proof of this inclusion is the following claim. Let
$\gamma=\min(C^*)$. 

\begin{claim}
Let $\gamma<\beta<\tau'_{\gamma}$ (see Lemma~\ref{l4} for the definition of
$\tau'_\gamma$). There is a $\xi_\beta\in C^*$ and a sequence 
 $$ \langle \zeta_\xi\mid\xi\in C^*-\xi_\beta\rangle $$ 
such that the following holds:
\begin{itemize}
\item[(a)] $\xi<\zeta_\xi<\tau_\xi$ and 
                $\pi^\mcT_{0,\xi}(\zeta_\xi)\ge\pi^\mcU_{\gamma,\xi}(\beta)$
                whenever  $\xi\in C^*-\xi_\beta$, and 
\item[(b)] $\pi^\mcU_{\xi,\xi'}\circ\pi^\mcT_{0,\xi}(\zeta_\xi)=
                  \pi^\mcT_{0,\xi'}(\zeta_{\xi'})$ 
                whenever $\xi<\xi'$ are in $C^*-\xi_\beta$. 
\end{itemize}
\end{claim}

\begin{proof}
Pick a sequence $\langle\bar{\zeta}_\xi\mid\xi\in\lim(C^*)\rangle$ such that
$\xi<\bar{\zeta}_\xi<\tau_\xi$ and
$\pi^\mcT_{0,\delta}(\bar{\zeta}_\xi)\ge\pi^\mcU_{\gamma,\xi}(\beta)$ for all
$\xi\in\lim(C^*)$. Using Lemma~\ref{u.l.pullback-1} we 
obtain a regressive function $g$ defined on $\lim(C^*)$ such that    
$(\pi^\mcU_{\bar{\xi},\xi})^{-1}\circ\pi^\mcT_{0,\xi}(\bar{\zeta}_\xi)
  \in\rng(\pi^\mcT_{0,\bar{\xi}})$ 
whenever $\xi\in\lim(C^*)$ and $\bar{\xi}\in[g(\xi),\xi)\cap C^*$. 
Using Fodor's lemma, we find a $\xi_\beta\in C^*$ and a stationary set
$A'_\beta\subseteq\lim(C^*)$ such that $g(\xi)=\xi_\beta$ for all 
$\xi\in A'_\beta$. Then, using the pigeonhole principle along with the fact that
$\tau'_{\xi_\beta}<\delta$, we find an ordinal $\zeta'_\beta$ such that
$\xi_\beta<\zeta'_\beta<\tau'_{\xi_\beta}$, 
and a stationary set $A_\beta\subseteq A'_\beta$ such that 
$\pi^\mcU_{\xi_\beta,\xi}(\zeta'_\beta)=\pi^\mcT_{0,\xi}(\bar{\zeta}_\xi)$ for
all $\xi\in A_\beta$. For $\xi\in C^*-\xi_\beta$ then let 
\[
\zeta_\xi=(\pi^\mcT_{0,\xi})^{-1}\circ\pi^\mcU_{\xi_\beta,\xi}(\zeta'_\beta).
\]
Notice that $\zeta_\xi$ is defined for every $\xi\in C^*-\xi_\beta$: Just pick
any $\xi'\in A_\beta$ such that $\xi'\ge\xi$; then 
\[
\pi^\mcU_{\xi_\beta,\xi}(\zeta'_\beta)=
(\pi^\mcU_{\xi,\xi'})^{-1}\circ\pi^\mcT_{0,\xi'}(\bar{\zeta}_\beta)
\]
and the right side is in the range of $\pi^\mcT_{0,\xi}$ by
Lemma~\ref{u.l.pullback-1}. (Notice that for $\xi\in A_\beta$ we have
$\zeta_\xi=\bar{\zeta}_\xi$.) Clause (a) of the Claim then follows
easily from our choice of ordinals $\bar{\zeta}_\xi$. Now if $\xi<\xi'$
are in $C^*-\xi_\beta$, then  
\[
\pi^\mcU_{\xi,\xi'}\circ\pi^\mcT_{0,\xi}(\zeta_\xi)=
\pi^\mcU_{\xi,\xi'}\circ\pi^\mcU_{\xi_\beta,\xi}(\zeta'_\beta)=
\pi^\mcU_{\xi_\beta,\xi'}(\zeta'_\beta)=
\pi^\mcT_{0,\xi'}(\zeta_{\xi'}),
\]
which verifies (b) of the Claim. \end{proof}

We now complete the proof of Lemma~\ref{u.l.pullback-2}. 
Let $\varepsilon=\cof^{\mbfV}(\tau'_{\gamma})$. 
Pick a sequence $\langle\beta_i\mid i<\varepsilon\rangle$
that is increasing and cofinal in $\tau'_{\gamma}$. Appealing to the above
Claim, for each $i<\varepsilon$ there is $\xi^i=\xi_{\beta_i}$
and a sequence $\langle\zeta^i_\xi\mid\xi\in C^*-\xi^i\rangle$ such that (a)
and (b) of the Claim hold with $\xi^i,\zeta^i_\xi$ in place of
$\xi_\beta,\zeta_\xi$. Let $\xi^*\in C^*$ be an upper bound for the $\xi^i$;
such a $\xi^*$ exists as $\varepsilon<\delta$. Then (a) and (b) of
the above Claim hold with $\xi^*,\zeta^i_\xi$ in place of
$\xi_\beta,\zeta_\xi$ whenever $i<\varepsilon$. Also, since
each $\pi^\mcT_{0,\xi}$ maps $\tau_\xi$ cofinally into $\tau'_\xi$, and
$\pi^\mcU_{\gamma,\xi}$ maps $\tau'_{\gamma}$ cofinally into $\tau'_\xi$, each
sequence $\langle\zeta^i_\xi\mid i<\varepsilon\rangle$ is cofinal in
$\tau_\xi$. It follows that for any $\xi<\xi'$ in $C^*-\xi^*$, the assignment
$\zeta^i_\xi\mapsto\zeta^i_{\xi'}$ maps a cofinal subset of $\tau_\xi$
cofinally into $\tau_{\xi'}$. 

Fix $\xi<\xi'$ in $C^*-\xi^*$ and an $i<\varepsilon$. Let $f,f'\in W$ be such
that $f$ is the $<_{E^W}$-least surjection of $\xi$ onto
$M\cut\zeta^i_\xi$ and $f'$ is the $<_{E^W}$-least surjection of
$\xi'$ onto $W\cut\zeta^i_{\xi'}$; recall that $E^W$ is the extender
sequence of $W$. Then $f\in W\cut\tau_\xi$ and $f'\in W\cut\tau_{\xi'}$. 
Appealing to (b) of the Claim and to the elementarity of the maps
$\pi^\mcT_{0,\xi}$ and $\pi^\mcU_{\xi,\xi'}$ we conclude that 
$\pi^\mcU_{\xi,\xi'}\circ\pi^\mcT_{0,\xi}(f)=\pi^\mcT_{0,\xi'}(f')$. Given
$x\in W\cut\zeta^i_\xi$, let $\eta<\xi$ be such that $x=f(\eta)$. Then 
\begin{eqnarray*}
\pi^\mcU_{\xi,\xi'}\circ\pi^\mcT_{0,\xi}(x)&=&
\pi^\mcU_{\xi,\xi'}\circ\pi^\mcT_{0,\xi}(f(\eta))=
\pi^\mcU_{\xi,\xi'}\circ\pi^\mcT_{0,\xi}(f)
   \left(\pi^\mcU_{\xi,\xi'}\circ\pi^\mcT_{0,\xi}(\eta)\right)\\
&=&
\pi^\mcT_{0,\xi'}(f')(\pi^\mcT_{0,\xi}(\eta))=
\pi^\mcT_{0,\xi'}(f')\left((\pi^\mcT_{0,\xi'}(\eta)\right)=
\pi^\mcT_{0,\xi'}(f'(\xi))\\
&=&
\pi^\mcT_{0,\xi'}(x).
\end{eqnarray*}
Here the first two equalities on the middle line follow from the fact that
$\crp(\pi^\mcU_{\xi,\xi'})=\xi$ and $\crp(\pi^\mcT_{\xi,\xi'})\ge\xi$. 

Since the sequence $\langle\zeta^i_\xi\mid i<\gamma\rangle$ is cofinal in
$\tau_\xi$, the computation in the previous paragraph yields 
$\pi^\mcU_{\xi,\xi'}\circ\pi^\mcT_{0,\xi}[W\cut\tau_\xi]
 \subseteq\pi^\mcT_{0,\xi'}[W\cut\tau_{\xi'}]$. 
As the assignment $\zeta^i_\xi\mapsto\zeta^i_{\xi'}$ is cofinal in
$\tau_{\xi'}$ (as we saw above), it follows that the set on the
left is a cofinal subset of the set on the right. This completes the 
proof of Lemma~\ref{u.l.pullback-2}. 
\end{proof}     

Without loss of generality we will assume that the conclusions of
Lemma \ref{u.l.pullback-2} hold for any pair of ordinals $\xi<\xi'$ in
$C^*$. For any such pair of ordinals we let 
\[
\pi_{\xi,\xi'}=
(\pi^\mcT_{0,\xi'})^{-1}\circ\pi^\mcU_{\xi,\xi'}\circ\pi^\mcT_{0,\xi}
\rst (W\cut\tau_\xi).
\]
Then $\pi_{\xi,\xi'}:W\cut\tau_\xi\to W\cut\tau_{\xi'}$ is a fully elementary
map such that 
\begin{itemize}
\item $\crp(\pi_{\xi,\xi'})=\xi$ and $\pi_{\xi,\xi'}(\xi)=\xi'$, and 
\item $\pi_{\xi,\xi'}$ maps $W\cut\tau_\xi$ cofinally into
$W\cut\tau_{\xi'}$. 
\item The system 
      $\langle W\cut\tau_\xi,\pi_{\xi,\xi'}\mid\xi<\xi'\mbox{ in }C^*\rangle$
      is commutative and continuous, so if $\xi$ is a limit point of $C^*$,
      then $W\cut\tau_\xi$ is the direct limit of the system 
      $\langle W\cut\tau_{\bar{\xi}},\pi_{{\bar{\xi}},\xi'}\mid
       \bar{\xi}<\xi'\mbox{ in }C^*\cap\xi\rangle$. 
\end{itemize}

The above conclusions follow easily from the definition of $\pi_{\xi,\xi'}$
and the previous calculations. We give some details on the continuity: Given a
limit point $\xi$ of $C^*$, pick some $\bar{\xi}\in C^*\cap\xi$. If $x\in
W\cut\tau_\xi$, pick some $\zeta<\tau_{\bar{\xi}}$ such that $x\in
W\cut\pi_{\bar{\xi},\xi}(\zeta)$; this is possible by the cofinality of the
map $\pi_{\bar{\xi},\xi}$. If $\bar{f}\in W\cut\tau_{\bar{\xi}}$ is a
surjection of $\bar{\xi}$ onto $W\cut\zeta$, then
$f=\pi_{\bar{\xi},\xi}(\bar{f})$ is a surjection of $\xi$ onto
$W\cut\pi_{\bar{\xi},\xi}(\zeta)$, so $x=f(\alpha)$ for some
$\alpha<\xi$. Now pick $\xi'\in C^*\cap[\bar{\xi},\xi)$ such that
$\alpha<\xi'$. Since $f$ and $\alpha$ belong to $\rng(\pi_{\xi',\xi})$, it follows that $x\in\rng(\pi_{\xi',\xi})$ 
as well.

\begin{lemma}\label{u.l.coherence}
Let $\xi<\xi'$ be in $C^*$. Let $F_{\xi,\xi'}$ be the $(\xi,\xi')$-extender
derived from $\pi_{\xi,\xi'}$. Then $F_{\xi,\xi'}$ is a $W$-extender that
coheres to $W$.  
\end{lemma}

\begin{proof}
We verify that $F_{\xi,\xi'}$ coheres to $W$; that $F_{\xi,\xi'}$ is a
$W$-extender is clear. As $W\cut\tau_\xi\models\zfc^-$, the ultrapower
$\ult(W\cut\tau_\xi,F_{\xi,\xi'})$ can be elementarily embedded into 
$W\cut\tau_{\xi'}$ in the usual way by assigning $\pi_{\xi,\xi'}(f)(a)$ to the
object in the ultrapower represented by $[a,f]$. This shows that the
ultrapower is well-founded. Let $Q$ be its transitive collapse, and let
$k:Q\to W\cut\tau_{\xi'}$ be the embedding defined as above. We then have
$\sigma\rst(\xi'+1)=\idm$. It follows that $k=\idm$, as otherwise $\crp(k)$
would be a cardinal in $Q$ larger than $\xi'$. Hence $Q$ is an initial segment
of $W\cut\tau_{\xi'}$. Since $\pi_{\xi,\xi'}$ maps
$\tau_\xi$ cofinally into $\tau_{\xi'}$ and $\pi_{\xi,\xi'}=k\circ i$, where
$i:Q\to W\cut\tau_{\xi'}$ is the ultrapower embedding, then $k$ maps $\On\cap Q$
cofinally into $\tau_{\xi'}$. It follows that $Q=W\cut\tau_{\xi'}$. \end{proof} 

Let 
\begin{equation} \label{u.e.dirlim-1}
\mbox{$\tilde{W}$ be the direct limit of the diagram $\langle
W\cut\tau_\xi,\pi_{\xi,\xi'}\mid\xi<\xi'\mbox{ in }C^*\rangle$}, 
\end{equation}
and
\begin{equation} \label{u.e.dirlim-2}
\mbox{$\pi_\xi:W\cut\tau_\xi\to\tilde{W}$ be the direct limit maps.}
\end{equation}

\begin{lemma}\label{u.l.elem}
Let $\theta$ be a regular cardinal larger than $\delta^+$, and let $X$ be an
elementary substructure of $H_\theta$ such that $\kappa=X\cap\delta\in\delta$
and  
\[
\langle W\cut\tau_\xi,\pi_{\xi,\xi'}\mid\xi<\xi'\mbox{ in }C^*\rangle\in X.
\] 
Let $H$ be the transitive collapse of $X$, and let $\sigma:H\to H_\theta$ be the
inverse of the Mostowski collapsing isomorphism. Then $\kappa\in C^*$,
$W\cut\tau_\kappa\in H$, $\sigma(W\cut\tau_\kappa)=\tilde{W}$, and 
$\sigma\rst(W\cut\tau_\kappa)=\pi_\kappa$. 
\end{lemma}

\begin{proof}
Since $C^*\in X$ and $C^*$ is club in $\delta$, the
critical point $\kappa$ of $\sigma$ is a limit point of $C^*$, hence an
element of $C^*$. Also $\sigma(C^*\cap\kappa)=C^*$. The inverse image of the
diagram 
$\langle W\cut\tau_\xi,\pi_{\xi,\xi'}\mid\xi<\xi'\mbox{ in }C^*\rangle$ under
$\sigma$ is thus of the form 
$\langle\bar{W}_\xi,\bar{\pi}_{\xi,\xi'}
 \mid\xi<\xi'\mbox{ in }C^*\cap\kappa\rangle$ where
 $\sigma(\bar{W}_\xi)=W\cut\tau_\xi$ for all $\xi\in C^*\cap\kappa$.  
Since the height of each $\bar{W}_\xi$ is strictly smaller than $\kappa$, the
map $\sigma$ is the identity on $\bar{W}_\xi$, so $\bar{W}_\xi=W\cut\tau_\xi$
and $\bar{\pi}_{\xi,\xi'}=\pi_{\xi,\xi'}$ for all $\xi<\xi'$ in
$C^*\cap\kappa$. Because the diagram 
$\langle W\cut\tau_\xi,\pi_{\xi,\xi'}\mid\xi<\xi'\mbox{ in }C^*\rangle$ is
continuous, the structure 
$W\cut\tau_\kappa=
 \lim\langle W\cut\tau_\xi,\pi_{\xi,\xi'}\mid
 \xi<\xi'\mbox{ in }C^*\cap\kappa\rangle$, being a direct limit of a diagram
 in $H$ is itself an element of $H$, and the same is also true of the direct
 limit  maps $\pi_{\xi,\kappa}$ where $\xi\in C^*\cap\kappa$. It follows that  
\begin{eqnarray*}
\sigma(W\cut\tau_\kappa)&=&
\sigma(\lim\langle W\cut\tau_\xi,\pi_{\xi,\xi'})
       \mid\xi<\xi'\mbox{ in }C^*\cap\kappa\rangle) \\ &=&
\lim\langle W\cut\tau_\xi,\pi_{\xi,\xi'}\mid\xi<\xi'\mbox{ in }C^*\rangle=
\tilde{W}.
\end{eqnarray*}

Furthermore, from the elementarity of $\sigma$, we obtain that
$\sigma(\pi_{\xi,\kappa})=\pi_\xi$ for all ordinals $\xi\in
C^*\cap\kappa$. Thus if $x\in W\cut\tau_\kappa$, then
$x=\pi_{\xi,\kappa}(\bar{x})$ for some $\xi\in C^*\cap\kappa$ and $\bar{x}\in
W\cut\tau_\xi$, hence  
\[
\sigma(x)=\sigma(\pi_{\xi,\kappa}(\bar{x}))=
\sigma(\pi_{\xi,\kappa})(\sigma(\bar{x}))=
\pi_{\xi,\delta}(\bar{x})=\pi_{\kappa,\delta}(\pi_{\xi,\kappa}(\bar{x}))=
\pi_{\kappa,\delta}(x),
\]
which shows that $\sigma\rst(W\cut\tau_\kappa)=\pi_\kappa$. \end{proof} 

It follows from the above lemma that the collection of all elementary
substructures $X$ of $H_\theta$ such that $\kappa=X\cap\delta\in\delta$ and
$\tilde{W}$ (from (\ref{u.e.dirlim-1})) collapses to $W\cut\kappa^{+W}$ under
the Mostowski collapsing isomorphism coming from $X$, is a club subset of
$\ptm_\delta(H_\theta)$. This is a strengthening of 
Mitchell-Schindler \cite[Lemma 3.5]{uemwlc}. 

We do not know if there is a significanly simpler proof of
Lemma~\ref{u.l.elem} based on the idea used in Mitchell-Schindler \cite[Lemma
3.5]{uemwlc}, as it seems that one needs to require $\tilde{W}\in X$ in order
to obtain the conclusions of Lemma~\ref{u.l.elem}. 

Given a regular cardinal $\theta>\delta^+$ and an elementary substructure $X$
of $H_\theta$ such that $X\cap\delta\in\delta$ and the diagram 
$\langle W\cut\tau_\xi,\pi_{\xi,\xi'}\mid\xi<\xi'\mbox{ in }C^*\rangle$ is an
element of $X$, define the following objects: 
\begin{itemize}
\item $H_X$ is the transitive collapse of $X$ and $\sigma_X:H_X\to X$ is the
      inverse of the Mostowski collapsing isomorphism. 
\item $\kappa_X=\crp(\sigma_X)$ and $\tau_X=\tau_{\kappa_X}=\kappa^{+W}_X$. 
\end{itemize}
If $Y$ is another such structure, $Y\supseteq X$ and $\kappa_Y>\kappa_X$ then:
\begin{itemize}
\item $\sigma_{X,Y}:H_X\to H_Y$ is defined by
      $\sigma_{X,Y}=\sigma^{-1}_Y\circ\sigma_X$. 
\item We write $\pi_X$ for
      $\pi_{\kappa_X}$ and $\pi_{X,Y}$ for $\pi_{\kappa_X,\kappa_Y}$.
\end{itemize}
Here of course all objects introduced above depend on $\theta$, but in our
applications we will keep $\theta$ fixed, so we treat it as a suppressed
parameter. It follows from the above lemma that
$\sigma_X\rst(W\cut\tau_X)=\pi_X$ and
$\sigma_{X,Y}\rst(W\cut\tau_X)=\pi_{X,Y}$, so these
restrictions depend only on the ordinals $\kappa_X$ and $\kappa_Y$ and not on
the structures $X$ and $Y$ themselves. In particular we get 
\begin{equation}\label{u.e.indep}
F_{X,Y}\dfeq F_{\kappa_X,\kappa_Y}=
\mbox{the $W$-extender at $(\kappa_X,\kappa_Y)$ derived from }\sigma_{X,Y},
\end{equation}
and
\begin{equation}
F_X \dfeq \mbox{the $W$-extender at $(\kappa,\delta)$ derived from }
         \pi_X=\sigma_X\rst(W\cut\tau_X).         
\end{equation}
We also note that if $\theta>\delta$ then $\ptm_\delta(H_\theta)$ is defined
consistently with our definition of $\ptm_\kappa(\lambda)$, that is, elements
$X$ of $\ptm_\delta(H_\theta)$ have the property $X\cap\delta\in\delta$. 

\begin{lemma}\label{u.l.s-theta}
Given a stationary set $S$ as in $(\ref{u.e.s})$, let $S^\theta$ be the
collection of all $X\in\ptm_\delta(H_\theta)$ satisfying the following
requirements:   
\begin{itemize}
\item[(a)] $X$ is an elementary substructure of $H_\theta$.
\item[(b)] The diagram 
           $\langle W\cut\tau_\xi,\pi_{\xi,\xi'}\mid
            \xi<\xi'\mbox{in }C^*\rangle$ 
           is an element of $X$. 
\item[(c)] $X\cap\delta^+\in S$.
\end{itemize}
Then $S^\theta$ is stationary, and for any $X,Y\in S^\theta$ such that
$X\subseteq Y$ and $\kappa_X<\kappa_Y$, the extender $F_{X,Y}$ is an element of
$\mbfM$. 
\end{lemma}

\begin{proof}
The stationarity of $S^\theta$ follows from the stationarity of $S$ and from
the fact that the collection of all $X$ satisfying clauses (a) and (b) is 
a club in $\ptm_\delta(H_\theta)$. We now verify that
$F_{X,Y}\in\mbfM$. For this, it will suffice to see that
$\pi_{X,Y}\in\mbfM$. We first observe that
$\pi_{X,Y}\rst\tau_X\in\mbfM$; for this we use the fact that 
$\pi_{X,Y}\rst\tau_X=(\pi^{-1}_Y\rst\tau_Y)\circ(\pi_X\rst\tau_X)$. To see
that both restriction maps on the right of this equality are in $\mbfM$, let
us argue for instance for $X$. Since $X\in S^\theta$ and $S\subseteq\mbfM$,
the restriction $\pi_X\rst\otp(X\cap\delta^+)=\sigma_X\rst\otp(X\cap\delta^+)$,
being the (unique) isomorphism between $\otp(X\cap\delta^+)$ and $X\cap\delta^+$, is an
element of $\mbfM$. The equality here comes from Lemma~\ref{u.l.elem}. As
$\tau_X\le\otp(X\cap\delta^+)$, the restriction $\pi_X\rst\tau_X$ 
is an element of $\mbfM$ as well. 

Now, since $W\cut\tau_X$ and $W\cut\tau_Y$ are both in $\mbfM$, it
suffices to argue that 
$\pi_{X,Y}$ is fully determined by $\pi_{X,Y}\rst\tau_X$. Since
$W\cut\tau_X\models\zfc^-$, the canonical well-ordering $<_W$ orders
$W\cut\tau_X$ in order-type $\tau_X$. So if $a\in W\cut\tau_X$, then there is
an ordinal $\alpha<\tau_X$ such that $a$ is the $\alpha$-th element of
$W\cut\tau_X$ under $<_W$. As $\pi_{X,Y}$ is elementary, $\pi_{X,Y}(a)$ is
the $\pi_{X,Y}(\alpha)$-th element of $W\cut\tau_Y$ under $<_W$. So the value
of $\pi_{X,Y}(a)$ is fully determined by the action of $\pi_{X,Y}$ on the
elements of $\tau_X$. 
\end{proof} 

\subsection{Proof of Theorem \ref{t0.1}(a)}

So far we completed task (A) in the strategy outlined at the beginning of this
section. We now proceed toward completing task (B). Our argument here
depends on the anti-large cardinal hypothesis we make in $\mbfV$. We start
with the case where no proper class inner model with a Woodin cardinal in the
sense of $\mbfV$ is allowed, that is, we are about to complete the proof of
Theorem~\ref{t0.1}(a). 

Here we apply a frequent extension argument, 
similar to that in Schim\-mer\-ling-Steel~\cite{mcm}; see also
Mitchell-Schimmerling~\cite{cwcc}. A general version of the frequent
extension argument in the simplified context of extender models below $0^\P$
is described in Zeman \cite[\S\S 7.5, 8.3]{imlc}, and also Cox
\cite{ctsr}. The layout of the argument in 
Cox \cite{ctsr} uses a system of extender models indexed by structures from a
suitable stationary set, rather then a linear chain of internally aproachable
structures. In our situation, it will be convenient to use this layout. The
initial step of the frequent extension argument thus produces 
a system of structures and embeddings indexed with elements $X$ of the
collection $S^\theta_\delta$ for a suitable $\theta$ (see below for the
definition of $S^\theta_\delta$); we denote these objects
by $\bar{N}_X,N^*_X,\rho_X,\tilde{\sigma}_X,\sigma'_X,\dots$ We
construct these objects in such a way that for each index $X$, each individual
$\bar{N}_X,N^*_X,\dots$ is an element of $\mbfM$. The entire indexed system, of
course, need not be an element of $\mbfM$. It may be the case (we
do not know) that one can proceed more liberally and run some constructions
in $\mbfV$ 
instead of $\mbfM$, but in the end we need to guarantee that certain premice
and phalanxes we construct are in $\mbfM$, and are iterable in the sense of
$\mbfM$. Our approach guarantees this automatically. Also, it is worth noting
that in order to run the frequent extension argument, we merely need the
anti-large cardinal hypothesis that no proper class inner model with a Woodin
cardinal exists in the sense of $\mbfM$. (So at this point in the proof of
Theorem~\ref{t0.1}(a) we do not need the full anti-large cardinal assumption
of the theorem.) 

Recall the set $S_\delta$ from (\ref{i.e.sd}). Working in $\mbfV$, pick a
regular $\theta$ such that $H_{\delta^+}\in H_\theta$. Following the notation 
introduced in Lemma~\ref{u.l.s-theta}, set
$S^\theta_\delta=(S_\delta)^\theta$. We show that for all
but nonstationarily many $X\in S^\theta_\delta$ and all $Y\in S^\theta_\delta$
such that $X\in Y$, the ultrapower $\ult(W,F_{X,Y})$ is well-founded and the
phalanx $(W,\ult(W,F_{X,Y}),\kappa_Y)$ is normally iterable in
the sense of $\mbfM$. (Under our anti-large cardinal hypothesis, this is
actually equivalent to its iterability in the sense of $\mbfV$). By
Lemma~\ref{cmt:abs} applied inside $\mbfM$, the 
extender $F_{X,Y}$ is an element of $W$. However, by
Lemma~\ref{u.l.coherence}, the extender induces the ultrapower map
$\pi_{X,Y}$ that maps $\tau_X$ cofinally into $\tau_Y$,
which is false in $W$. This is a contradiction, and
completes the proof of Theorem~\ref{t0.1}(a).

\begin{remark} 
It is the use of the frequent extension argument that necessitates  
the assumption of stationarity of $S_\delta$ in place of the weaker assumption
that $\ptm_\delta(\delta^+)\cap\mbfM$ is stationary. In general, the
assumption that $\ptm_\delta(\delta^+)\cap\mbfM-S_\delta$ is stationary is not
sufficient to run the frequent extension argument; see for
instance R\"asch-Schindler \cite{ncp}.
\end{remark}

Before we proceed, recall that given phalanxes $(P,P',\alpha)$ and
$(Q,Q',\beta)$, a pair of maps $(\sigma,\sigma')$ is an embedding of
$(P,P',\alpha)$ into $(Q,Q',\beta)$ if and only if $\sigma:P\to P'$ and
$\sigma':Q\to Q'$ are $\Sigma_0$-elementary and cardinal preserving,
$\sigma\rst\alpha=\sigma'\rst\alpha$, $\sigma[\alpha]\subseteq\beta$, and
$\sigma'(\alpha)\ge\beta$. We will only deal with embeddings that are fully
elementary, but of course one may consider embeddings of various degrees of
preservation. 

If $\mcT$ is a normal iteration tree on $(P,P',\alpha)$ according to the 
uniqueness strategy, $\alpha$ is a cardinal in $P'$ and $(Q,Q',\beta)$ is
iterable via some iteration strategy $\Sigma^Q$, we can copy $\mcT$
onto a normal iteration tree on $(Q,Q',\beta)$ according to $\Sigma^Q$ via the
embedding $(\sigma,\sigma')$. Note that we will 
use this terminology also in the case where $Q=Q'$, that is, when we talk
about an embedding of a phalanx into a premouse. 

For each $X\in S^\theta_\delta$, pick some $Y=Y(X)\in S^\theta_\delta$ such
that $X\in Y$. Then let $G_X=F_{X,Y(X)}$ and
$\lambda_X=\kappa_{Y(X)}$. We will prove that for all but nonstationarily many
$X\in S^\theta_\delta$,  
\begin{equation}\label{e.u.iterable}
\mbox{$\ult(W,G_X)$ is well-founded and 
$(W,\ult(W,G_X),\lambda_X)$ is normally iterable.}
\end{equation} 
Here we talk about iterability in the sense of $\mbfM$.

Heading for a contradiction, assume there is a stationary 
\begin{equation}\label{u.e.bad-s}
S\subseteq S^\theta_\delta
\end{equation}
such that either $\ult(W,G_X)$ is ill-founded or else this ultrapower is
well-founded but the phalanx $(W,\ult(W,G_X),\lambda_X)$ 
is not normally iterable in the sense of $\mbfM$. So there is some 
$\theta_0>\delta^{+\mbfV}$ such that $\theta_0$ is a successor cardinal in
$W$ and the same is true with $W\cut\theta_0$ in place of $W$ for all $X\in
S$. Write $N$ for $W\cut\theta_0$ and $N'_X$ for $\ult(N,G_X)$, where we have identified 
well-founded parts with their transitive collapses. Thus even if $N'_X$ is 
ill-founded, it is well-founded past $\lambda_X$, and if $N'_X$ is well-founded,
then it is actually transitive. For each $X\in S$ let $\mcT_X\in\mbfM$ be a
putative normal iteration tree on the phalanx $(N,N'_X,\lambda_X)$ witnessing
that the phalanx is not normally iterable in $\mbfM$. Here we mean that either
$\mcT_X$ has a last ill-founded model, or $\mcT_X$ is of limit length and has
no cofinal well-founded branch. If $N'_X$ is ill-founded, we let
$\mcT_X$ be the phalanx $(N,N'_X,\lambda_X)$ whose last model is $N'_X$, which
offers uniform treatment of the situations listed above. 

Pick a $\mbfV$-regular $\theta^*>\theta$ such that $\theta,S,\mcT^X\in
H_{\theta^*}$ whenever $X\in S$. For each $X$, construct an elementary
substructure $Z_X\prec H_{\theta^*}^\mbfM$ such that 
\begin{itemize}
\item $Z_X\in\mbfM$ and is countable in $\mbfM$, and  
\item $G_X,\mcT_X,\tilde{\tau}\in Z_X$, where $\tilde{\tau}=\delta^{+\tilde{W}}$.
\end{itemize} 

Let $H^Z_X$ be the transitive collapse of $Z_X$, let 
$\rho_X:H^Z_X\to H^\mbfM_{\theta^*}$ 
be the inverse to the associated Mostowski collapsing isomorphism, and let 
 $$ (\bar{\mcT}_X,\bar{N}_X,\bar{N}'_X,\bar{\lambda}_X)=\rho_X^{-1}(\mcT_X,N,N'_X,\lambda_X). $$
So each of these objects individually is in $\mbfM$, although the collection consisting
of all of them may not be an element of $\mbfM$. 

Next we do a pressing down argument which is a typical part of any frequent
extension argument. Consider 
$X\in S$ such that $H^\mbfM_{\delta^+}\in X$. Recall that since
$X\cap\delta^+\in\mbfM$, the restriction 
$\sigma_X\rst\tau_X$ is an element of $\mbfM$ as well. Since $\delta$ is a
cardinal and $\sigma_X\rst\tau_X$ maps $\tau_X<\delta$ cofinally into
$\tilde{\tau}$, necessarily $\tilde{\tau}<\delta^{+\mbfM}$, so there is a
surjection $f:\delta\to\tilde{\tau}$ such that $f\in\mbfM$. As 
$\tilde{\tau},H^\mbfM_{\delta^+}\in X$, there is a surjection
$f:\delta\to\tilde{\tau}$ such that $f\in\mbfM\cap X$. Since 
$\sigma_X[Z_X\cap\tau_X]\subseteq{\mbfM}\cap X$, it follows that 
$\sigma_X[Z_X\cap\tau_X]\subseteq f[\kappa_X]$.

Since $\sigma_X\rst\tau_X$ is an element of $\mbfM$, so is $\sigma_X[Z_X\cap\tau_X]$.
In $\mbfM$, the set $\sigma_X[Z_X\cap\tau_X]$ is a countable subset of
$\tilde{\tau}$. Since $\cof^\mbfM(\kappa_X)$ is uncountable in $\mbfM$, we can
find an ordinal $\alpha<\kappa_X$ such that  
$\sigma_X[Z_X\cap\tau_X]\subseteq f[\alpha]$. But since $f\in X$, also 
$f[\alpha]$ is an element of $X$. Thus, letting $a=f[\alpha]$, the set $X$
witnesses the existential quantifier in the statement 
\begin{equation}\label{u.e.down}
H_{\theta^*}\models(\exists v\in S)(a\in v).
\end{equation}
Obviously the set 
\[
S'=\{X\in S\mid\mbox{$H^\mbfM_{\delta^+}\in X$ and 
$(\exists X'\prec H_{\theta^*})(S\in X'\;\&\;X'\cap H_\theta=X)$}\} 
\] 
is stationary. Given $X\in S'$, let $X'\prec H_{\theta^*}$ witness
this. Since $a,S\in X'$, applying the elementarity of $X'$ to  
(\ref{u.e.down}) yields
\[
X'\models(\exists v\in S)(a\in v).
\] 
So there is some $\bar{X}\in S\cap X'$ with $a\in\bar{X}$, and such $\bar{X}$
is obviously an element of $X$. Since 
$\sigma_X[Z_X\cap\tau_X]\subseteq a\subseteq\bar{X}$ (the latter inclusion
follows from the facts that $a\in\bar{X}$, the cardinality of $a$ is less than 
$\delta$, and $\bar{X}\cap\delta$ is transitive), we
can define a regressive function $g:S'\to S$ such that
$\sigma_X[Z_X\cap\tau_X]\subseteq g(X)$ for all $X\in S'$. By pressing down,
there are a stationary $S^*\subseteq S'$ and an $X^*\in S$ such that 
\begin{equation}\label{u.e.fix}
\sigma_X[Z_X\cap\tau_X]\subseteq X^*\in X\mbox{ for all }X\in S^*. 
\end{equation}

Fix the following notation: 
\begin{itemize}
\item $H^*$ is the transitive collapse of $X^*$. 
\item $\sigma^*_X:H^Z_X\to H^*$ is a partial map defined by
      $\sigma^*_X=\sigma^{-1}_{X^*,X}\circ\rho_X$. 
\item $\tau^*_X=\sup(\sigma^*_X[\bar{\tau}_X])$, where 
      $\bar{\tau}_X=\rho^{-1}_X(\tau_X)$.
\end{itemize}

It follows from (\ref{u.e.fix}) that
$\rho_X[\bar{\tau}_X]=Z_X\cap\tau_X\subseteq\sigma_{X^*,X}[\tau_{X^*}]$, so
$\sigma^*_X(\xi)$ is defined for all $\xi<\bar{\tau}_X$. Furthermore, since
$\rho_X,\sigma_{X^*,X}\rst\tau_{X^*}$ are both elements of $\mbfM$, so is
$\sigma^*_X\rst\bar{\tau}_X$. An argument similar to that in the proof of
Lemma~\ref{u.l.s-theta} shows that $\sigma^*_X\rst(\bar{N}_X\cut\bar{\tau}_X)$
is fully determined by $\sigma^*_X\rst\bar{\tau}_X$, and is defined on the
entire $\bar{N}_X\cut\bar{\tau}_X$.
Hence  $\sigma^*_X\rst(\bar{N}_X\cut\bar{\tau}_X):
 \bar{N}_X\cut\bar{\tau}_X\to W\res\tau^*_X=N\res\tau^*_X$ is a
$\Sigma_0$-preserving embedding mapping $\bar{N}_X\cut\bar{\tau}_X$ cofinally
into $N\res\tau^*_X$. (Recall that $N\res\tau^*_X$ is the same as
$N\cut\tau^*_X$ without the top extender; and since $\bar{\tau}_X$, being a
cardinal in $\bar{N}_X$, does not index an extender in $\bar{N}_X$, under a
slight abuse of notation we have
$\bar{N}_X\cut\bar{\tau}_X=\bar{N}_X\res\bar{\tau}_X$.) With a tiny bit of
effort one can see that $\sigma^*_X\rst(\bar{N}_X\cut\bar{\tau}_X)$ is fully
elementary (note that $W\res\tau^*_X$ is an elementary substructure of
$W\res\tau_{X^*}$). We thus have the following conclusion: 
\begin{equation}\label{u.e.sigma-star}
\sigma^*_X\in\mbfM\mbox{ and }
\sigma_{X^*,X}\circ\sigma^*_X\rst(\bar{N}_X\cut\bar{\tau}_X)=
\rho_X\rst(\bar{N}_X\cut\bar{\tau}_X).
\end{equation}

This agreement of
$\sigma_{X^*,X}\circ\sigma^*_X$ with $\rho_X$ on $\bar{N}_X\cut\bar{\tau}_X$
makes it possible to run the argument in the proof of the Interpolation Lemma
(Zeman \cite[Lemma 3.6.10]{imlc}), and obtain an acceptable structure
$N^*_X$ extending $N\res\tau^*_X$ and embeddings 
$\tilde{\sigma}_X:\bar{N}_X\to N^*_X$ and $\sigma'_X:N^*_X\to N$ with
the following properties: 
\begin{itemize}
\item $N^*_X$ is an end-extension of $N\res\tau^*_X$. 
\item $\tau^*_X=\kappa_{X^*}^{+N^*_X}$ hence 
      $\ptm(\kappa_{X^*})\cap N^*_X= \ptm(\kappa_{X^*})\cap
      N\res\tau^*_X\subseteq\ptm(\kappa_{X^*})\cap W$. 
\item $\tilde{\sigma}_X$ is an extension of $\sigma^*_X\rst\bar{N}_X$, and 
      $\sigma'_X$ is an extension of $\sigma_{X^*,X}\rst(N\res\tau^*_X)$. 
\item Both maps $\tilde{\sigma}_X$ and $\sigma'_X$ are 
      $\Sigma_0$-preserving, and cardinal preserving.
\item $\sigma'_X\circ\tilde{\sigma}_X=\rho_X\rst\bar{N}_X$. 
\item $N^*_X,\sigma'_X$ and $\tilde{\sigma}_X$ are elements of $\mbfM$. 
\end{itemize}

Briefly, $N^*_X$ is constructed as the ultrapower (Zeman \cite{imlc} uses
the term ``pseudoultrapower") of $\bar{N}_X$ by $\sigma^*_X$, the map
$\tilde{\sigma}_X$ is the ultrapower embedding, and $\sigma'_X$ is the
factor map between $\tilde{\sigma}_X$ and $\rho_X\rst\bar{N}_X$. The last
item follows from the fact that the entire ultrapower construction takes place
in $\mbfM$. Since $N$ is an initial segment of an 
extender model below its successor cardinal, $N$ is a passive premouse and
$N\models\zfc^-$. It follows that both maps are in fact fully elementary, and
$N^*_X$ is a passive premouse. Also, we can view $\sigma'_X$ as a map from
$N^*_X$ into $W$, and in this case $\sigma'_X$ is $\Sigma_0$-preserving and
cardinal preserving.    

The above construction gives us the following for all $X,Y\in S^*$ such that
$Y=Y(X)$ or $Y(X)\in Y$; note that it makes sense to form
$\ult(N^*_X,F_{X^*,Y})$, as follows from the second bullet point in the 
paragraph below (\ref{u.e.sigma-star}):
\begin{equation}\label{u.e.reflection} 
\begin{array}{l}
\mbox{Either $\ult(N^*_X,F_{X^*,Y})$ is ill-founded,}\\ 
\mbox{or else $(W,\ult(N^*_X,F_{X^*,Y}),\kappa_Y)$ is not iterable.}
\end{array}
\end{equation}

The conclusion for the case where $Y(X)\in Y$ follows immediately from the
conclusion for $Y=Y(X)$ since the pair $(\idm,k)$ is an embedding of the phalanx
$(W,\ult(W,F_{X^*,Y(X)}),\lambda_X)$ into $(W,\ult(W,F_{X^*,Y}),\kappa_Y)$
where $k$ is the factor map between the two ultrapower embeddings. To see the
conclusion for $Y=Y(X)$, recall that we put $\mcT_X$, an iteration tree
witnessing the non-iterability of the phalanx
$(N,\ult(N,F_{X,Y}),\lambda_X)$, into $Z_X$. Then an argument similar to the 
proof of Steel \cite[Lemma 2.4(b)]{cmip} yields that $\bar{\mcT}_X$ witnesses
that $(\bar{N}_X,\bar{N}'_X,\bar{\lambda}_X)$ is not iterable in the sense of
$\mbfM$; here recall that $\bar{N}'_X=\ult(\bar{N}_X,\bar{G}_X)$ where
$\bar{G}_X=\rho_X^{-1}(G_X)$. (We stress that we are assuming $Y=Y(X)$.) Here
is the place where we use the fact that $\theta_0$ is a successor cardinal in
$W$; see Steel \cite[Lemma 6.13]{cmip} for details concerning why this
assumption is useful. 

Let $k':\bar{N}'_X\to\ult(N^*_X,F_{X^*,Y})$ be the map defined by 
\[
k':[a,f]_{\bar{G}_X}\mapsto[\rho_X(a),\tilde{\sigma}_X(f)]_{F_{X^*,Y}} 
\]
whenever $a\in[\bar{\lambda}_X]^{<\omega}$ and $f\in\bar{N}_X$ is such that
$\dom(f)=[\bar{\kappa}_X]^{|a|}$, where $\bar{\kappa}_X=\crp(\bar{G}_X)$.  
Then the pair $(\rho_X\rst\bar{N}_X,k')$ is an embedding of the phalanx
$(\bar{N}_X,\bar{N}'_X,\bar{\lambda}_X)$ into
$(W,\ult(N^*_X,F_{X^*,Y}),\lambda_X)$ witnessing (\ref{u.e.reflection}) in the
case where $Y=Y(X)$. This follows by the standard 
computation: First notice that it follows from the definition of $k'$, for any
$a\in[\bar{\lambda}_X]^{<\omega}$ we have
$k'(a)=k'([a,\idm]_{\bar{G}_X})=[\rho_X(a),\idm]_{F_{X^*,Y}}=\rho_X(a)$ and
$k'(\bar{\lambda}_X)=\lambda_X$, so we have the necessary agreement between
the two maps. To see that $k'$ is sufficiently elementary, given a formula
$\varphi(v)$, \L o\'s theorem yields $\bar{N}'_X\models\varphi([a,f])$ iff 
\[
x=\{u\in[\bar{\kappa}_X]^{|a|}\mid\bar{N}_X\models\varphi(f(u))\}
  \in(\bar{G}_X)_a.
\]
By applying $\rho_X$, we have $\rho_X(x)\in (G_X)_{\rho_X(a)}$ or,
equivalently, $\rho_X(a)\in\sigma_{X,Y}(\rho_X(x))$. Since
$\sigma_{X,Y}(\rho_X(x))=\sigma_{X,Y}\circ\sigma'_X\circ\sigma^*_X(x)=
 \sigma_{X^*,Y}(\sigma^*_X(x))$, 
we conclude that $\sigma^*_X(x)\in(F_{X^*,Y})_{\rho(a)}$. Now
$$ \sigma^*_X(x)=\{u\in[\kappa_{X^*}]^{|a|}\mid
 N^*_X\models\varphi(\tilde{\sigma}_X(f)(u))\} $$ 
and, by another application of \L o\'s theorem, we have
 $$ \ult(N^*_X,F_{X^*,Y})\models\varphi([\rho_X(a),\tilde{\sigma}_X(f)]. $$
This completes the proof of (\ref{u.e.reflection}). 

From now on, the argument closely follows Mitchell-Schimmerling
\cite[\S 3]{cwcc}. Write $\kappa^*$ for $\kappa_{X^*}$. 
We introduce a relation $\stl$ on premice $Q$ such that $(W,Q,\kappa^*)$ is an
iterable phalanx:
\begin{equation}\label{u.e.stl}
\parbox{4in}
{
$Q'\stl Q$ iff there is a normal iteration tree $\mcT$ on
$(W,Q,\kappa^*)$ such that $Q'$ is an initial segment of the last model
$M^\mcT_\infty$  of $\mcT$, and, letting $b$ be the main branch of $\mcT$, one
of the following holds: 
\begin{itemize}
\item $W$ is on $b$. 
\item $Q$ is on $b$ and there is a truncation point on $b$.
\item $Q$ is on $b$, there is no truncation point on $b$, and $Q'$ is a proper
      initial segment of $M^\mcT_\infty$. 
\end{itemize}
}
\end{equation}

The following lemma is a simple instance of 
Mitchell-Schimmerling \cite[Lemma 3.2]{cwcc}, rephrased for our purposes. 

\begin{lemma}\label{u.l.minimal}
Let $\mathfrak{Q}$ be the transitive closure of the singleton $\{W\}$ under the
relation $\stl$. Then $\stl$ restricted to $\mathfrak{Q}$ is well-founded. 
\end{lemma}

\begin{proof}
See the proof of Mitchell-Schimmerling \cite[Lemma 3.2]{cwcc}. \end{proof}  

The above lemma says that if we start with the phalanx $(W,W,\kappa^*)$ and
build a linear chain of iteration trees $\mcT_i$ such that $\mcT_0$ is on
$(W,W,\kappa^*)$ and $\mcT_{i+1}$ is on $(W,Q_{i+1},\kappa^*)$, where $Q_{i+1}$
is the last model of $\mcT_i$, then for some finite $n$ the model $Q_{n+1}$ is
on the main branch of $\mcT_n$ and there is no truncation on the branch. 

Pick $\tilde{X}\in S^*$ and let 
\begin{equation*}
S^{**}=\{Z\in S^*\mid Y(\tilde{X})\in Z\}.
\end{equation*}
Note that $X^*\in Z$ whenever $Z\in S^{**}$. Lemma~\ref{u.l.minimal} is
formulated for $\mbfV$, but we apply it in $\mbfM$ for each individual 
$X\in S^{**}$. Given $X\in S^{**}$, either $\ult(W,F_{X^*,X})$ is ill-founded
or else $(W,\ult(W,F_{X^*,X}),\kappa_X)$ is not iterable. This follows from
our choice of $S$ in (\ref{u.e.bad-s}), where the requirement is made for
$G_{\tilde{X}}$ and $\lambda_{\tilde{X}}$ in place of $F_{X^*,X}$ and
$\kappa_X$, respectively. However, for $X\in S^{**}$ we have
$F_{X^*,X}=F_{\tilde{X},X}\circ G_{\tilde{X}}\circ F_{X^*,\tilde{X}}$ where we
view extenders as maps, which, together with some little effort, which we
leave to the reader, yields the above conclusion (the argument in the first two
paragraph of the proof of Lemma~\ref{u.l.absolute-minimal} can be used to get
the conclusion for the composition $F_{\tilde{X},X}\circ G_{\tilde{X}}$; the
presence of $F_{X^*,\tilde{X}}$ requires an additional argument, which is
neverheless easy). By Lemma~\ref{u.l.minimal} applied in $\mbfM$, for 
every $X\in S^{**}$ there is a $\stl$-minimal $Q_X$ such that the phalanx
$(W,Q_X,\kappa^*)$ is iterable, and either $\ult(Q_X,F_{X^*,X})$ is
ill-founded or else $(W,\ult(Q_X,F_{X^*,X}),\kappa_X)$ is not iterable. So
$Q_X\in\mbfM$, and the notion of iterability is also in the sense of $\mbfM$. 
The next lemma says that 
it is possible to find a structure $X\in S^{**}$ such that $Q_X$ can replace
$Q_Y$ for club many structures $Y$. 

\begin{lemma}\label{u.l.absolute-minimal}
There is an $X\in S^{**}$ such that for all $Y\in S^{**}$ with $X\in Y$, the
following hold in $\mbfM$:
\begin{itemize}
\item[(a)] Either $\ult(Q_X,F_{X^*,Y})$ is ill-founded or else  
           $(W,\ult(Q_X,F_{X^*,Y}),\kappa_Y)$ is not iterable. 
\item[(b)] The premouse $Q_X$ is $\stl$-minimal, that is, if $Q'\stl Q_X$
           then {\rm(a)} fails with $Q'$ in place of $Q_X$. 
\end{itemize}
\end{lemma}

\begin{proof}
Notice that if $X\in Y$, and either $\ult(Q_X,F_{X^*,X})$ is ill-founded, or 
 $$ (W,\ult(Q_X,F_{X^*,X}),\kappa_X) $$ 
is not iterable, then (a) holds, that is, either 
$\ult(Q_X,F_{X^*,Y})$ is ill-founded, or else the phalanx 
 $$ (W,\ult(Q_X,F_{X^*,Y}),\kappa_Y) $$ 
is not iterable. 

This follows from the fact that the extender $F_{X^*,Y}$ is ``larger" than
$F_{X^*,X}$. More precisely, let 
 $$ k:\ult(Q_X,F_{X^*,X})\to \ult(Q_X,F_{X^*,Y}) $$ 
be the factor map. Then the pair $(\idm,k)$ is an
embedding of the phalanx $(W,\ult(Q_X,F_{X^*,X}),\kappa_X)$ into the phalanx 
$(W,\ult(Q_X,F_{X^*,Y}),\kappa_Y)$. Informally, since $Y$ is ``larger" than $X$,
the structure $Q_X$ is still ``bad" in the sense that it satisfies (a), but
need not be $\stl$-minimal. 
 
Assuming the lemma is false, for every $X\in S^{**}$ there is some $Y\in
S^{**}$ with $X\in Y$, 
such that $Q_X$ is not minimal such that (a) holds. Letting $X_0=X$,
construct inductively a sequence $\langle X_i\mid i\in\omega\rangle$ such that 
$X_i\in X_{i+1}\in S^{**}$ and $Q_{X_{i+1}}\stl Q_{X_i}$. That is, there is an
iteration tree $\mcT_i\in\mbfM$ on $(W,Q_{X_i},\kappa)$ such that
$Q_{X_{i+1}}$  is the last model of $\mcT_i$, and either $W$ is on the main
branch of $\mcT_i$, or else $Q_{X_i}$ is on the main branch of $\mcT_i$, in
which case either there is a truncation on the main branch of $\mcT_i$, or else
$Q_{X_{i+1}}$ is a proper initial segment of the last model of $\mcT_i$. The
sequence $\langle Q_{X_i}\mid i\in\omega\rangle$ witnesses that the relation
$(\stl)^\mbfM$ is ill-founded in the sense of $\mbfV$. Notice that, by
minimality, each $Q_{X_i}$ is a set. By absoluteness of
well-foundedness, $(\stl)^\mbfM$ is ill-founded in the sense of $\mbfM$, a
contradiction. \end{proof}  

Now pick $X$ as in Lemma~\ref{u.l.absolute-minimal}, and $Y\in S^{**}$ such that
$X^*,X,Y(X)\in Y$. Working in $\mbfM$, compare the phalanxes
$(W,Q_X,\kappa^*)$ and $(W,N^*_X,\kappa^*)$. Let $\mcU$ and $\mcV$ be the
iteration trees arising in the comparison. That is, $\mcU$ is on
$(W,Q_X,\kappa^*)$. 

Because $W$ witnesses the soundness of
$\mbfK^\mbfM\cut\delta^+$, the standard argument shows that $W$ cannot be on
the main branches of both trees. We now get a final contradiction by ruling
out all other possibilities. 

We begin with the observation that $Q_X$ must be on the main branch of
$\mcU$. Otherwise, $N^*_X$ is on the main branch of $\mcV$, as follows from the
previous paragraph. Furthermore, since $W$ is on the main branch of $\mcU$, and
$W$ computes successor cardinals correctly on a stationary class of
cardinals, there is no truncation on the main branch of $\mcV$. Letting $R$ be
the last model on $\mcV$, we have $R\stl Q_X$, which by
Lemma~\ref{u.l.absolute-minimal} means that $\ult(R,F_{X^*,Y})$ is
well-founded and $(W,\ult(R,F_{X^*,Y}),\kappa_Y)$ is iterable. On the other
hand, by (\ref{u.e.reflection}) we have that either $\ult(N^*_X,F_{X^*,Y})$ is
ill-founded or else $(W,\ult(N^*_X,F_{X^*,Y}),\kappa_Y)$ is not
iterable. Moreover, let 
 $$ k:\ult(N^*_X,F_{X^*,Y})\to\ult(R,F_{X^*,Y}) $$ 
be the map defined by $k([a,f])=[a,\pi^{\mcV}(f)\rst[\kappa^*]^{|a|}]$, 
where $\pi^\mcV:N^*_X\to R$ is the
map along the main branch of $\mcV$. Then the pair $(\idm,k)$ is an embedding of
the phalanx $(W,\ult(N^*_X,F_{X^*,Y}),\kappa_Y)$ into
$(W,\ult(R,F_{X^*,Y}),\kappa_Y)$. This is a contradiction, which completes the proof 
that $Q_X$ must be on the main branch of $\mcU$.

Let $Q$ be the last model on $\mcU$, and let $R$, as above, be the last model on
$\mcV$. We next argue that there is no truncation on the main branch of $\mcU$
and $Q\unlhd R$. Otherwise, a simple discussion by cases yields that there
is no truncation on the main branch of $\mcV$, and $R\unlhd Q$. Moreover, $R\stl Q_X$. 
Notice also that $Q_X$ is a set, because any 
counterexample to iterability is witnessed by a set-sized initial segment of the model. 
From this it follows that $N^*_X$ is on the main branch of $\mcV$. 
Now it is easy to see that we can get a contradiction as in the
previous paragraph.  

To summarize, we are left with the following situation: 
\begin{itemize}
\item $Q_X$ is on the main branch of $\mcU$.  
\item There is no truncation on the main branch of $\mcU$, and $Q\unlhd R$. 
\end{itemize}
Recall that the pair $(\idm,\sigma'_X)$ is an embedding of the phalanx
$(W,N^*_X,\kappa^*)$ into the model $W$. The map $\sigma'_X$ is defined below
(\ref{u.e.sigma-star}), and when viewed as a map from $N^*_X$ into $W$, it is
$\Sigma_0$-preserving and cardinal-preserving. (Here we again use the fact
that $\theta_0$ was chosen to be a cardinal in $W$.) Let $\mcV'$ be the copy
of $\mcV$ via the pair $(\idm,\sigma'_X)$, let $R'$ be the last model of
$\mcV'$, and let $\sigma':R\to R'$ be the copy map. Also, let $\pi^\mcU:Q_X\to
Q$ be the iteration map along the main branch of $\mcU$. 
Now $R\res\tau_R=N^*_X\res\tau_R$, where
$\tau_R=(\kappa^*)^{+R}$ and, by the copying construction,
$\sigma'\rst(R\res\tau_R)=\sigma'_X\rst(R\res\tau_R)$. It follows that the
extender on $R$ derived from $\sigma'_X$ is compatible with $F_{X^*,X}$ and
measures all sets in $\ptm(\kappa^*)\cap Q=\ptm(\kappa^*)\cap Q_X$, since
there is no drop in the main branch of $\mcU$. 

Let $\tilde{R}=\sigma'(Q)$; if $Q=R$, we of course let $\tilde{R}=R'$. Define
 a map $k:\ult(Q_X,F_{X^*,X})\to \tilde{R}$ by  
\[
 k:[a,f]\mapsto\sigma'\circ\pi^\mcU(f)(a)
\]
for $a\in[\kappa_X]^{<\omega}$ and $f\in Q_X$ such that
$\dom(f)=[\kappa^*]^{|a|}$.

Using \L o\'s theorem, a straightforward computation similar to that in
the proof of (\ref{u.e.reflection}) shows that $k$ is
$\Sigma_0$-preserving. It is also easy to see that $k$ is cardinal 
preserving. Immediately from the definition of $k$ we see that 
$k\rst\kappa_X=\idm$. It follows that the pair $(\idm,k)$ is an embedding of
the phalanx $(W,\ult(Q_X,F_{X^*,X}),\kappa_X)$ into the phalanx
$(W,\tilde{R},\kappa_X)$. Since the iteration indices in $\mcV$ are above 
$\kappa^*$, the iteration indices in $\mcV'$ are above
$\sigma'_X(\kappa^*)=\kappa_X$, so $(W,\tilde{R},\kappa_X)$ is a $W$-generated
phalanx in the sense of Steel \cite[Definition 6.7]{cmip}. As $W$ is 
embeddable into $\Kc$, the phalanx $(W,\tilde{R},\kappa_X)$ is embeddable into a
$\Kc$-generated phalanx, and is therefore iterable (see Steel
\cite[Lemma 6.9]{cmip}). It follows that also
$(W,\ult(Q_X,F_{X^*,X}),\kappa_X)$ is iterable, which contradicts the
definition of $Q_X$.  

Thus, assuming that $\ult(W,G_X)$ is ill-founded or else 
$(W,\ult(W,G_X),\lambda_X)$ fails to be iterable for all $X\in S$ where $S$
comes from (\ref{u.e.bad-s}), we arrived at a final contradiction, thereby
completing the proof of Theorem~\ref{t0.1}(a). 

\subsection{Proof of Theorem \ref{t0.1}(b)}

Finally we turn to the proof of Theorem~\ref{t0.1}(b). Here we work under the
assumption that $0^\P$ does not exist in ${\mbfM}$, which means
that, if we use premice with Jensen's indexing of non-overlapping extenders,
as described in 
Zeman~\cite[Chapter 8]{imlc}, then the comparison process can be carried out
using just linear iterations. In particular, $\mbfK^\mbfM$ is always linearly
iterable in the sense of $\mbfV$, no matter how large premice exist in
$\mbfV$. Linear iteration trees along with the condensation properties of such
premice allow us, for an extender with support below $\delta^+$, to reduce the
question of well-foundedness and iterability of $\ult(\mbfK^\mbfM,F)$ to that
of well-foundedness and iterability of $\ult(\mbfK^\mbfM\cut\delta^+,F)$,
thereby avoiding any direct use of the frequent extension argument. We note
here that since we deal with linear iterations, the fact that
$\ult(\mbfK^\mbfM,F)$ is well-founded already implies that
$\ult(\mbfK^\mbfM,F)$ is fully fine structurally iterable. We will
give the details of such a reduction below. In what
follows we will refer to Zeman~\cite[Chapter 8]{imlc} on the core model theory
below $0^\P$.   

So assume $\delta\ge\omega_1$ and the trees $\mcT,\mcU$ come from the
coiteration of $\mbfK^\mbfM$ against $N$. Thus we let $W=\mbfK^\mbfM$ in this
case; otherwise we keep the rest of the notation as at the beginning of
this section. There is a slight non-uniformity which comes into the argument
when dealing with the cases $\delta>\omega_1$ and $\delta=\omega_1$. 

If $\delta>\omega_1$ the length of the trees $\mcT$ and $\mcU$ is $\delta$,
as we assume, toward a contradiction, that $N$ out-iterates
$\mbfK^\mbfM\cut\delta$. In the case where $\delta=\omega_1$, our hypothesis
reads that $\mbfK^\mbfM\cut\omega_2$ fails to be universal with respect to
countable mice in $\mbfV$, so in this case we assume $N$ iterates past
$\mbfK^\mbfM\cut\omega_2$.  

For the rest of the argument, write $\delta^*=\max(\delta,\omega_2)$. Notice
however that even in the case 
where $\delta=\omega_1$, the extenders $F_{\xi,\xi'}$ are defined as in
Lemma~\ref{u.l.coherence}, that is, using only the trees $\mcT\rst\delta$ and
$\mcU\rst\delta$. Refering back to Lemma~\ref{u.l.coherence}, pick
$\kappa<\lambda$ in $C^*$. Here $C^*\subseteq\delta$, hence $\lambda<\delta$. 
Recall also that, consistently with the notation fixed below
(\ref{u.e.bad-s}), we write $W_\alpha$ for $M^\mcT_\alpha$ and $N_\alpha$ for
$M^\mcU_\alpha$.  

\begin{lemma}\label{u.l.iterability-end}
Let $F^*_{\kappa,\lambda}$ be the extender at $(\kappa,\lambda)$ on $N_\kappa$
derived from $\pi^\mcU_{\kappa,\lambda}$. Then
$\tilde{N}=\ult(N_{\delta^*},F^*_{\kappa,\lambda})$ is well-founded and
iterable. 

It follows that
$\ult(N_{\delta^*}\cut\delta^*,F^*_{\kappa,\lambda})$ is well-founded and
iterable.  
\end{lemma}

\begin{proof}
By the definition of the set $C^*$ there is an iteration map
$$ \pi^\mcU_{\kappa,\delta}:N_\kappa\to N_\delta. $$ 
Write $\sigma$ for
$\pi^\mcU_{\kappa,\delta}$, and use $\sigma$ to copy the linear iteration tree
$\mcU\rst[\kappa,\delta^*)$ onto a linear iteration tree $\mcU'$ on
$N_\delta$. This is easy to do, as the iteration trees are linear 
and the extenders used in these trees are internal. We thus obtain a copy map
$\sigma':N_{\delta^*}\to N'_{\delta^*}$ where we write $N'_\alpha$ for
$M^{\mcU'}_\alpha$. Notice that if $\delta=\omega_1$, there may be a truncation
on the tree $\mcU\rst[\delta,\delta^*)$, but it will not have any negative
effect on the argument below. Since $\sigma$ is $\Sigma^*$-preserving, so is
$\sigma'$. By the choice of $\kappa$, the critical points in the tree 
$\mcU\rst[\kappa,\delta^*)$ are at least $\kappa$, and the iteration indices are
at least $\tau'_\kappa=\kappa^{+N_\kappa}$ (this notation is consistent with
that fixed in Lemma~\ref{l4}), as $\mcU\rst[\kappa,\delta)$ does
not involve any truncation. It is also easy to see that
$\tau'_\kappa=\kappa^{+N_{\delta^*}}$. It follows that
$N_\kappa\cut\tau'_\kappa=N_{\delta^*}\cut\tau'_\kappa$ and 
$\sigma'\rst(N_{\delta^*}\cut\tau'_\kappa)=
 \sigma\rst(N_\kappa\cut\tau'_\kappa)$.  
In particular, the extender at $(\kappa,\lambda)$ derived from $\sigma'$ is
precisely $F^*_{\kappa,\lambda}$. This tells us that the ultrapower
$\ult(N_{\delta^*},F^*_{\kappa,\lambda})$ can be embedded into the premouse
$N'_{\delta^*}$ via the natural factor embedding
$k:[a,f]_{F^*_{\kappa,\lambda}}\mapsto\sigma'(f)(a)$ which is
$\Sigma^*$-preserving, as both $\sigma'$ and the ultrapower embeddings
are. Now $N'_{\delta^*}$, being an iterate of an iterable premouse, is
itself iterable, which guarantees the iterability of
$\ult(N_{\delta^*},F^*_{\kappa,\lambda})$. The last conclusion of the lemma is
an immediate consequence. \end{proof} 

We are now going to use a variant of the ``shift lemma" that will allow us to
embed $\ult(\mbfK^\mbfM\cut\delta^*,F_{\kappa,\lambda})$ into
$\ult(N_{\delta^*}\cut\delta^*,F^*_{\kappa,\lambda})$. Given structures
$Q,Q'$, a $\Sigma_0$-preserving map $\sigma:Q\to Q'$, extenders $G$ at
$(\kappa,\lambda)$ on $Q$ and $G'$ at $(\kappa,\lambda')$ on $Q'$ such that
$\sigma[\kappa]\subseteq\kappa$, and an order preserving map
$\mfk:\lambda\to\lambda'$, we write $(\sigma,\mfk):(Q,G)\to(Q',G')$ if and
only if for every $a\in[\lambda]^{<\omega}$ and 
$x\in\ptm([\kappa]^{|a|})\cap Q$, 
\[
x\in G_a\;\Longrightarrow\;\sigma(x)\cap[\kappa]^{|a|}\in G'_{\mfk[a]}.
\]
This corresponds to the similar notion discussed in Zeman
\cite[\S 2.5]{imlc}. The slight difference between the notion used here and
that in Zeman~\cite{imlc} is that, in our situation, it may happen that
$\sigma(\crp(G))>\crp(G')$. If $(\sigma,\mfk):(Q,G)\to(Q',G')$, then we can run
the proof of the shift lemma and show that there is precisely one
$\Sigma_0$-preserving embedding 
\[
\sigma':\ult(Q,G)\to\ult(Q',G') 
\]
such that $\sigma'\rst\lambda=\mfk$ and
$\sigma'\circ\pi_G=\pi_{G'}\circ\sigma$, where $\pi_G,\pi_{G'}$ are the
corresponding ultrapower embeddings. Moreover, if 
$Q\models\zfc^-$ and $\sigma$ is fully elementary, so is $\sigma'$. 

\begin{lemma}\label{u.l.iterability-front}
Let $\kappa<\lambda$ be in $C^*$, and let $F^*_{\kappa,\lambda}$ be as in
Lemma~\ref{u.l.iterability-end}. Then the following hold:
\begin{itemize}
\item[(a)] $(\pi^\mcT_{0,\delta^*}\rst(\mbfK^\mbfM\cut\delta^*),
             \pi^\mcT_{0,\delta^*}\rst\lambda):
            (\mbfK^\mbfM\cut\delta^*,F_{\kappa,\lambda})\to
            (N_{\delta^*}\cut\delta^*,F^*_{\kappa,\lambda})$. 
\item[(b)] $\ult(\mbfK^\mbfM\cut\delta^*,F_{\kappa,\lambda})$ is well-founded
           and iterable. 
\end{itemize}
\end{lemma}

\begin{proof}
Clause (a) follows by a straightforward computation using the fact that
$\crp(\pi^\mcT_{\lambda,\delta^*})\ge\lambda$. Recall that we are using
Jensen's indexing of non-overlapping extenders. This 
computation uses elementary facts about this indexing. Clause (b) follows from 
(a) and Lemma~\ref{u.l.iterability-end}. \end{proof} 

Now let $S=\ptm_\delta(\delta^+)\cap\mbfM$, pick $\theta$ large regular in
$\mbfV$, and let $S^\theta$ be as in Lemma~\ref{u.l.s-theta}. If we pick $X\in
Y$ such that both $X,Y\in S^\theta$, and let $\kappa=\kappa_X$ and
$\lambda=\kappa_Y$, then $F_{\kappa,\lambda}\in\mbfM$. From now on, we work
entirely in $\mbfM$. By Lemma~\ref{u.l.iterability-front},
$\ult(\mbfK^\mbfM\cut\delta^*,F_{\kappa,\lambda})$ is well-founded and
iterable. We show that this conclusion can be extended to $\mbfK^\mbfM$. 

\begin{lemma}\label{u.l.iterability-k}
$\ult(\mbfK^\mbfM,F_{\kappa,\lambda})$ is well-founded and iterable. 
\end{lemma}

\begin{proof}
Assume the contrary. Working in $\mbfM$, let $\nu$ be a $\mbfK^\mbfM$-cardinal
large enough such that, letting $Q=\mbfK^\mbfM\cut\nu$, the ultrapower
$\ult(Q,F_{\kappa,\lambda})$ is either ill-founded or not iterable,
and let $\mcV$ be a linear iteration tree on $Q$ in $\mbfM$ witnessing
this. Again, if $\ult(Q,F_{\kappa,\lambda})$ is ill-founded, we let $\mcV$ be
the tree of length $1$ consisting of this ill-founded ultrapower, in order to
treat the two cases uniformly. Let $\theta^*$ be a large regular
cardinal such that $\mcV\in H^\mbfM_{\theta^*}$, and let $Z$ be an elementary
substructure of $H^\mbfM_{\theta^*}$ of size less than $\delta$ such that
$F_{\kappa,\lambda},\mcV\in Z$ and $\bar{\delta}=Z\cap\delta\in\delta$. Also,
let $H$ be the transitive collapse of $Z$, and let $\sigma:H\to
H^\mbfM_{\theta^*}$ 
be the inverse to the Mostowski collapsing isomorphism. Notice that
$\bar{\delta}>\lambda$ and $\sigma(F_{\kappa,\lambda})=F_{\kappa,\lambda}$. 
Letting $(\bar{Q},\bar{\mcV})=\sigma^{-1}(Q,\mcV)$, the linear iteration tree
$\bar{\mcV}$ witnesses the ill-foundedness/non-iterability of
$\ult(\bar{Q},F_{\kappa,\lambda})$. One can of course focus merely on the
proof of well-foundedness, as for linear iterability, well-foundedness implies
iterability (as already mentioned above). However, the proof is the same for
both, and does not need to rely on this fact. We will use the following
observation which, as was pointed out to us by the referee, can also be proved
for premice with Mitchell-Steel indexing:   

\begin{claim}
Let $\mcR,\bar{\mcR}$ be the linear iteration trees on $Q,\bar{Q}$
respectively, coming from the comparison of $Q$ against $\bar{Q}$. Then the
critical points in $\bar{\mcR}$ are above~$\bar{\delta}$. 
\end{claim}

\begin{proof}
If $\bar{\delta}$ is not overlapped by an extender in $\bar{Q}$ then this is
immediate. So suppose $\bar{\delta}$ is overlapped. Here we apply a condensation
result particular for the indexing of the extenders we are using here; recall
this is the indexing described in Zeman \cite[Chapter 8]{imlc}. It follows that
either $\bar{Q}$ is a proper initial segment of $Q$ or else there is
$\mu<\bar{\delta}$ and $\beta>\bar{\delta}$ such that 
$E^{\bar{Q}}_\beta=\varnothing$, $\crp(E^{\bar{Q}}_\alpha)=\mu$ whenever
$\mu^{+\bar{Q}}\le\alpha<\beta$ and $\alpha$ indexes an extender in $\bar{Q}$,
and for every $\alpha>\beta$ that indexes an extender in $\bar{Q}$ we have
$\crp(E^{\bar{Q}}_\alpha)>\beta$. In the former case the comparison of $Q$
against $\bar{Q}$ is trivial, so let us focus on the latter case. 

Since $\sigma:\bar{Q}\to Q$ is elementary
and $\crp(\sigma)=\bar{\delta}$ then, by Zeman \cite[Lemma 8.2.2]{imlc},
$E^{\bar{Q}}\rst\beta=E^Q\rst\beta$. It follows that in the coiteration of $Q$
against $\bar{Q}$, all iteration indices used in $\bar{\mcR}$ are larger than
$\beta$, but by the discussion above, also the corresponding critical points
are above $\beta$, and hence above $\bar{\delta}$. 
\end{proof} 

We can now complete the proof of the lemma. Since $\delta^*\ge\omega_2^\mbfM$,
the initial segment $\mbfK^\mbfM\cut\delta^*$ is universal for all premice in
$\mbfM$ of size less than $\delta^*$; this follows by adjusting the proof  
for Steel's core model in Schimmerling-Steel~\cite{mcm} to the indexing we are
currently using. Hence $\mbfK^\mbfM\cut\delta^*$ iterates past $\bar{Q}$, and 
the coiteration terminates after less than $\delta^*$ steps. Let $\mcR'$ be
the linear iteration tree on the $\mbfK^\mbfM\cut\delta^*$-side of the
coiteraton of $\bar{Q}$ against $\mbfK^\mbfM\cut\delta^*$. Since
$\mbfK^\mbfM\cut\delta^*\unlhd Q$ 
and $\delta^*$ is a regular cardinal, the iteration tree on the $\bar{Q}$-side
of this coiteration is precisely $\bar{\mcR}$, and the iteration tree $\mcR'$
is essentially the same as $\mcR$ with the only difference that
$M^{\mcR'}_\alpha=M^\mcR_\alpha\cut\delta^*$ for all $\alpha<\lht(\mcR)$. 
By the above discussion, there is an iteration map
$\pi^{\bar{\mcR}}_\infty:\bar{Q}\to M^{\bar{\mcR}}_\infty$  where
$M^{\bar{\mcR}}_\infty$ is the last model on $\bar{\mcR}$. Moreover,  
$M^{\bar{\mcR}}_\infty\unlhd M^{\mcR'}_\infty$. By the Claim,
$\crp(\pi^{\bar{\mcR}}_\infty)\ge\bar{\delta}>\lambda$. It follows that
$\ult(\bar{Q},F_{\kappa,\lambda})$ can be embedded into
$\ult(M^{\bar{\mcR}}_\infty,F_{\kappa,\lambda})$ by
$[a,f]_{F_{\kappa,\lambda}}\mapsto[a,\pi^{\bar{\mcR}}_\infty(f)]$, which shows
that $\ult(M^{\bar{\mcR}}_\infty,F_{\kappa,\lambda})$ is either ill-founded, or
else not iterable. It follows that $\ult(M^{\mcR'}_\infty,F_{\kappa,\lambda})$
is ill-founded or not iterable. 

On the other hand, $Q^*=\ult(\mbfK^\mbfM\cut\delta^*,F_{\kappa,\lambda})$
is well-founded and iterable by Lemma~\ref{u.l.iterability-front}(b).  Let
$\rho$ be the associated ultrapower map. Use $\rho$ to copy the linear
iteration tree $\mcR'$ onto a linear iteration tree $\mcR^*$ on $Q^*$ via the
map $\rho$. Let $\rho^*:M^{\mcR'}_\infty\to M^{\mcR^*}_\infty$ be the copy map
between the last two models in the copy construction. Since the iteration
indices of $\mcR'$ are above $\bar{\delta}$, the map $\rho^*$ agrees with
$\rho$ below $\bar{\delta}$. It follows that the extender on
$M^{\mcR'}_\infty$ at $(\kappa,\lambda)$ derived from $\rho^*$ is precisely
$F_{\kappa,\lambda}$, hence $\ult(M^{\mcR'}_\infty,F_{\kappa,\lambda})$ can be
embedded into $M^{\mcR^*}_\infty$ via the natural map
$[a,f]_{F_{\kappa,\lambda}}\mapsto\rho^*(f)(a)$. As $\mcR^*_\infty$, being an
iterate of an iterable premouse, is itself iterable, this proves the
iterability of $\ult(M^{\mcR'}_\infty,F_{\kappa,\lambda})$. Now we have
reached a contradiction with the conclusion of the previous paragraph, thereby
completing the proof of Lemma~\ref{u.l.iterability-k}.  
\end{proof}

We can now finish the proof of Theorem~\ref{t0.1}(b). Here we could use a
criterion on absorption of extenders by $\mbfK$ which can be found on page~274
in Chapter~8 in Zeman~\cite{imlc} and which, as pointed out by the referee,
needs some corrections. However, to keep the text more uniform, we use a
variant of Lemma~\ref{cmt:abs} for the current indexing. Let
$\mbfK'=\ult(\mbfK^\mbfM,F_{\kappa,\lambda})$. Then $\mbfK^\mbfM$ coiterates
with $\mbfK'$ above $\lambda$, as otherwise $\mbfK$ would have an extender
overlapping $\lambda$ which would imply that also $N_\delta$ has such an
extender; see the paragraph immediately above
Lemma~\ref{u.l.iterability-end} for notation. Our indexing however does not
allow this, as 
this would mean that in at least one of the premice $N_i$ there is an overlap
of extenders. We now coiterate $\mbfK^\mbfM$ against $\mbfK'$ and, as the
critical points of the coiteration are above $\lambda$, similarly as
in the proof of Lemma~\ref{cmt:abs} we conclude that
$F_{\kappa,\lambda}\in\mbfK^M$. As in that proof, this gives a contradiction
since $F_{\kappa,\lambda}$ codes a map which collapses 
$\lambda^{+\mbfK^\mbfM}$ in $\mbfK^\mbfM$. This completes the proof of
Theorem~\ref{t0.1}(b).  

\subsection{An improvement suggested by the referee}\label{ss.referees-version}

The following improvement of Theorem~\ref{t0.1}, which we record as
Theorem~\ref{u.t.main-imp}, was suggested by the referee. We first introduce
some terminology, also suggested by the referee. 

Let $\delta$ be regular and $P$ be a premouse of height $\delta$. We say that
$P$ is neatly universal with respect to a class of premice $\mcS$ iff for
every $N\in\mcS$ there is a terminal comparison $(\mcT,\mcU)$ of $P$ against
$N$ such that $\mcM^\mcU_\infty$ is an initial segment of $\mcM^\mcT_\infty$,
there is no drop on the main branch of $\mcU$, the comparison uses only
extenders of length $<\delta$, and if the comparison is of length $\delta$
then $i^\mcU[\delta]\subseteq\delta$ where $i^\mcU$ is the iteration map along
the main branch of $\mcU$. 

We say that a premouse $N$ is short iff $N$ is $1$-small and, for some
$n\in\omega$, $N$ is $n$-sound and for
every $n$-maximal normal $Q$-structure guided iteration tree $\mcT$ on $N$ of
limit length, the $Q$-structure for $\mcT$ exists and is of the form
$J_\alpha(\mcM(\mcT))$ for some ordinal $\alpha$. We say that a short premouse
is short tree normally iterable iff $N$ is normally iterable with respect to
the $Q$-structure guided strategy.  

Let $\delta$ be regular. A premouse $R$ of height $\delta$ is anomalous iff
$R$ has a largest cardinal $\mu$ which is singular in $R$ and, letting
$\kappa=\cof^R(\mu)$, there are unboundedly many $\alpha<\kappa^{+R}$ such
that $\crp(E^R_\alpha)=\kappa$. The point of introducing the notion of
anomalous premouse is that if $P$ is an anomalous premouse which has a total
extender on its sequence with critical point $\cof^P(\mu_P)$ where $\mu_P$
witnesses that $P$ is anomalous then the ultrapower by such an extender has
height strictly larger than $\delta$. This may cause potential problems with
standard arguments when attempting a proof of universality of $R$.  

\begin{theorem}\label{u.t.main-imp}
Let $\delta>\omega$ be regular, and let $\mbfM$ be a proper class inner model
of $\zfc$. Then the following hold.
\begin{itemize}
\item[(a)] Assume that there is no proper class inner model with a Woodin
           cardinal from the point of view of $\mbfM$, and $S_\delta$ is
	   stationary in the sense of $\mbfV$.  Then
	   $\mbfK^\mbfM\cut\delta$ is neatly universal with respect to all
	   $(\delta+1)$-short tree normally iterable short premice of height
	   $<\delta$ in $\mbfV$. Moreover, if $\mbfK^\mbfM\cut\delta$ is not
	   anomalous, then ``$<\delta$" can be replaced with ``$\le\delta$".  
\item[(b)] Assume $\mbfM$ believes that $0^\P$ does not exist, and
           $\ptm_\delta(\delta^+)\cap\mbfM$ is stationary. If
	   $\delta>\omega_1^\mbfM$ then $\mbfK^\mbfM\cut\delta$ is neatly
	   universal with respect to all $(\delta+1)$-iterable premice of height
	   $<\delta$, and if $\mbfK^\mbfM\cut\delta$ is not anomalous then
	   ``$<\delta$" can be replaced with ``$\le\delta$". If
	   $\delta=\omega_1^\mbfM$ (hence $=\omega_1^\mbfV$) then
	   $\mbfK^\mbfM\cut\omega_2^M$ is universal with respect to all
	   countable premice, and if $\mbfK^\mbfM\cut\omega_2^M$ is not
	   anomalous then ``countable" can be replaced with ``of height
	   $\le\omega_1$".  
\end{itemize}
\end{theorem}

In the case of (a), our argument can be run as the comparison of $N$ against
$\mbfK^\mbfM\cut\delta$ is successful for premice as in (a) under the weaker
hypothesis that $\mbfM\models\mbox{``There is no proper class inner model with
a Woodin cardinal"}$, and a closer inspection of our frequent extension
argument reveals that it merely used $\delta>\omega_1^\mbfM$ rather than
$\delta>\omega_1$. 

In the case of (b), if $\delta>\omega_1$ all we really used in the proof of
Theorem~\ref{t0.1}(b) was that $\delta^*>\omega_1^\mbfM$. (See the paragraph
immediately preceding Lemma~\ref{u.l.iterability-end} for the definition of
$\delta^*$.)  If $\delta=\omega_1^\mbfM$ we similarly only used that
$\delta^*\ge\omega_2^\mbfM$. There was another place in the proof where we
referred to $\omega_2$, namely when we argued that $N$ iterates past
$\mbfK^\mbfM\cut\omega_2$, but one can easily check that the argument goes
through if we merely know that $N$ iterates past
$\mbfK^\mbfM\cut\omega_2^\mbfM$.

\section{Discussion}\label{sec:dis}

In this section we discuss some issues concerning why we are unable to
prove our main result in its full generality, and outline a possible approach
that may lead to removing all our restricting assumptions, at least in the
absence of Woodin cardinals. Our conjecture thus can be formulated as follows: 

\begin{conjecture}\label{d.conj}
Assume that there is no proper class inner model with a Woodin cardinal. Let
$\delta$ be an uncountable regular cardinal, and let $\mbfM$ be a proper class
inner model such that $\ptm_{\delta}(\delta^+)\cap\mbfM$ is stationary. Then
the following hold:
\begin{itemize}
\item[(a)] If $\delta>\omega_1^\mbfM$, then $\mbfK^\mbfM\cut\delta$ is
           universal for  
           all iterable premice in $\mbfV$ of cardinality less than $\delta$. 
\item[(b)] If $\delta=\omega_1^\mbfM$ (hence $\omega_1^\mbfM=\omega_1^\mbfV$),
           then $\mbfK^\mbfM\cut\omega_2$ is 
           universal for all countable iterable premice in $\mbfV$. 
\end{itemize}
\end{conjecture}

Theorem~\ref{t0.1}(b) says that the above conclusions (a) and (b) hold under
the rather restrictive anti-large cardinal hypothesis that a sharp for an
inner model with a strong cardinal does not exist. Unfortunately, we do not
know how to extend our proof of Theorem~\ref{t0.1}(b) to the more general 
situation with no inner models with Woodin cardinals. We do, however, have
some partial results along these lines that we proceed to describe.

\subsection{The case $\delta>\omega_1^\mbfM$}\label{ss.dgo1}
 
One drawback of our proof of Theorem~\ref{t0.1}(a) is the direct use of the
frequent extension argument. As we have already explained above, the use of
this argument requires us to work with the set $S_\delta$ from (\ref{i.e.sd})
in place of $\ptm_\delta(\delta^+)\cap\mbfM$, as it relies on the fact that
countable subsets of 
structures in $S_\delta$ can be internally covered. And, as explained before,
there is no way around, by the results in R\"asch-Schindler \cite{ncp} and
references 
therein. Our proof of Theorem~\ref{t0.1}(b) avoids the direct use of
the frequent extension argument by showing that for suitable $\delta$ and
$X,Y$, (i) $\ult(\mbfK^\mbfM\cut\delta,F_{X,Y})$ is well-founded and iterable,
and (ii) the question of well-foundedness and iterability of
$\ult(\mbfK^\mbfM,F_{X,Y})$ can be reduced to that of the well-foundedness and
iterability of $\ult(\mbfK^\mbfM\cut\delta,F_{X,Y})$. Here, (ii) relies on the
result from Schimmerling-Steel \cite{mcm} that $\mbfK^\mbfM\cut\delta$ is
universal for iterable premice of size less than $\delta$ in $\mbfM$ whenever
$\delta\ge\omega_2^\mbfM$. Under the assumption ``there is no proper class
inner model with a Woodin cardinal", (ii) needs to be replaced with a stronger
statement, namely  that $\ult(W,F_{X,Y})$ is well-founded and the phalanx
$(W,\ult(W,F_{X,Y}),\kappa_Y)$ is iterable, where $W$ is the extender model
witnessing the soundness of a suitably long initial segment of
$\mbfK^\mbfM$. 

In the following, we outline a way of generalizing the approach from the proof
of Theorem~\ref{t0.1}(b) for $\delta>\omega_1$ to the situation where we
assume there are no proper class inner models with a Woodin cardinal. The
argument from the proof of Theorem~\ref{t0.1}(b) can be generalized to obtain
the following:   

\begin{equation}\label{d.e.bottom}
\begin{array}{l}\ult(\mbfK^\mbfM\cut\delta,F_{X,Y})\mbox{ is well-founded, and }\\
(\mbfK^\mbfM\cut\delta,\ult(\mbfK^\mbfM\cut\delta,F_{X,Y}),\kappa_Y)\mbox{ is
      iterable.}\end{array}
\end{equation}

However, we do not know how to use (\ref{d.e.bottom}) along with the
universality of $\mbfK^\mbfM\cut\delta$ to prove the iterability
of the phalanx $(W,\ult(W,F_{X,Y}),\kappa_Y)$. 

The conclusion (\ref{d.e.bottom}) is established using a copying construction
that generalizes that in the proof of Theorem~\ref{t0.1}(b). More precisely,
letting $\kappa=\kappa_X$, if $\mcU$ is the normal iteration tree on $N$
coming from the coiteration with $W$ as described below (\ref{u.e.s}), we can
``embed" $\mcU\rst\kappa$ into $\mcU$ using a system of maps $\sigma_i$ where 
$\sigma_\alpha:M^\mcU_\alpha\to M^\mcU_\alpha$ is the identity map for
$\alpha<\kappa$, and $\sigma_\kappa:M^\mcU_\kappa\to M^\mcU_\delta$ is the
iteration map $\pi^\mcU_{\kappa,\delta}$. A straightforward, but a bit tedious 
computation shows that we can copy the iteration tree $\mcU$ onto a normal
iteration tree $\mcU'$ of length $\delta+\delta$ extending $\mcU$, where the 
copying construction uses the maps $\sigma_i$ instead of a single map. The
point here is to verify that there are no ``conflicts" 
when we copy, that is, $\mcU'$ is indeed a normal iteration tree. The
situation can be then depicted by the following diagram; we write $N_\alpha$
for $M^\mcU_\alpha$ and $N'_\alpha$ for $M^{\mcU'}_\alpha$.

\begin{displaymath}
\xymatrix{
\mcU' & N  \ar@{.>}[r] & 
                N'_h \ar@{.>}[r]  &
		N'_\kappa \ar@{.>}[r]  &  
                N'_\delta \ar@{.>}[r]  & 
                N'_{\delta+\alpha} \ar@{.>}[r]  &   
	        N'_{\delta+\delta} \\
\mcU & N \ar@{.>}[r] \ar[u]|{\sigma_0=\idm} & 
               N_h \ar@{.>}[r]  \ar[u] |{\sigma_h=\idm} & 
               N_\kappa \ar@{.>}[r] 
               \ar[ur]|{\sigma_\kappa=\pi^\mcU_{\kappa,\delta}} & 
               N_i \ar@{.>}[r] & 
               N_{\kappa+\alpha} \ar@{.>}[r]  \ar[u]|{\sigma_{\kappa+\alpha}}
	       & N_\delta \ar[u]|{\sigma_\delta} \ar[r]|{\pi_{F^*_{X,Y}}} &
	       \tilde{N}\ar[ul]|{\tilde{\sigma}} 
              }
\end{displaymath}

As the iteration indices in $\mcU\rst[\kappa,\delta)$
are larger than $\kappa$, the copy map $\sigma_\delta$ agrees with
$\sigma_\kappa$ on $\ptm(\kappa)\cap N_\kappa=\ptm(\kappa)\cap N_\delta$. It
follows that the two derived extenders agree, and in fact they agree with the
map $\sigma_X\rst(\ptm(\kappa)\cap N_\delta)$. (See the
settings above (\ref{u.e.indep})). So if we pick $Y$ such that $X\in Y$, and
let $\lambda=\kappa_Y$ and $F^*_{X,Y}$ be the $(\kappa,\lambda)$-extender
derived from $\sigma_{X,Y}\rst(\ptm(\kappa)\cap N_\delta)$, we see that
$\tilde{N}=\ult(N_\delta,F^*_{X,Y})$ is well-founded, and the factor map 
$\tilde{\sigma}:\tilde{N}\to N'_{\delta+\delta}$ is the identity
below $\lambda$, and sends $\lambda$ to $\delta$. 

It follows that the pair
$(\idm,\tilde{\sigma})$ is an embedding of the phalanx
$(N'_{\delta+\delta},\tilde{N},\lambda)$ into 
$N'_{\delta+\delta}$, which proves the iterability of the phalanx. This
of course yields the iterability of the phalanx
$(N'_{\delta+\delta}\cut\delta,\tilde{N}\cut\delta,\lambda)$. 
Note that $N'_{\delta+\delta}\cut\delta=N_\delta\cut\delta$. Recall that
$\mcT$ is the iteration tree on the $W$-side of the coiteration of $N$ with
$W$; see again the settings below (\ref{u.e.s}). 
As in the proof of Theorem~\ref{t0.1}(b), we can then show that the phalanx
$(\mbfK^\mbfM\cut\delta,\ult(\mbfK^\mbfM\cut\delta,F_{X,Y}),\lambda)$ can be
elementarily embedded into
$(N'_{\delta+\delta}\cut\delta,\tilde{N}\cut\delta,\lambda)$ using the pair of
maps $(\pi^\mcT_{0,\delta}\rst(\mbfK^\mbfM\cut\delta),k)$ where   
$k:\ult(\mbfK^\mbfM\cut\delta,F_{X,Y})\to\tilde{N}\cut\delta$ is the unique
elementary map that agrees with $\pi^\mcT_{0,\delta}$ below $\lambda$ and
satisfies $k\circ\pi_{X,Y}=\pi^*_{X,Y}\circ\pi^\mcT_{0,\delta}$, and
$\pi^*_{X,Y}$ and $\pi_{X,Y}$ are the ultrapower embeddings coming from
$\ult(N_\delta\cut\delta,F^*_{X,Y})$ and
$\ult(\mbfK^\mbfM\cut\delta,F_{X,Y})$, respectively. This completes
the sketch of the proof of (\ref{d.e.bottom}). 

\subsection{The case $\delta=\omega_1^\mbfM=\omega_1^\mbfV$} 

Although superficially this case may appear to be very special, it is
interesting in many respects. The main difference between this case and the
one discussed in Subsection~\ref{ss.dgo1} 
is that, as we already mentioned in the introduction, we cannot expect
$\mbfK^\mbfM\cut\delta$ to be universal for countable mice, by the result of
Jensen \cite{spmwv}. Suppose we try to prove Conjecture~\ref{d.conj} using the
strategy from Subsection~\ref{ss.dgo1}. By the fact that $\mbfK^\mbfM\cut\delta$
is no longer universal for countable premice, establishing (\ref{d.e.bottom})
is not sufficient to run the argument. It seems that what we would need here
is the variant of (\ref{d.e.bottom}) where $\delta$ is replaced with
$\delta^+$, and this in turn requires to use the coiteration of $W$
against $N$ which has length $\delta^+$. However, the extender $F_{X,Y}$ is
determined already by the coiteration of length $\delta$. If we are in the 
situation where no sharp for an inner model with a strong cardinal exists, all
iteration trees can be made linear, so once the well-foundedness of
$\ult(N_{\delta^+},F^*_{X,Y})$ has been established, one can easily embed
$\ult(\mbfK^\mbfM\cut\delta^+,F_{X,Y})$ into $\ult(N_{\delta^+},F^*_{X,Y})$
since
$\pi^\mcT_{0,\delta^+}=\pi^\mcT_{\delta,\delta^+}\circ\pi^\mcT_{0,\delta}$ and
the critical point of $\pi^\mcT_{\delta,\delta^+}$ is at least $\delta$. In
the situation with many strong cardinals, we have non-linear iteration trees,
so even if we succeeded in running an argument similar to that outlined in
Subsection~\ref{ss.dgo1} and proved that
$(N_{\delta^+},\ult(N_{\delta^+},F^*_{X,Y}),\kappa_Y)$ is 
iterable, the argument for linear iterations for ``pulling" back to
$\mbfK^\mbfM$ we have just mentioned can be mimicked here only if we 
know (among other things) that $\delta$ is on the branch $[0,\delta^+]_\mcT$,
or, more generally, if some sufficiently large element of the set $C^*$
defined below Lemma~\ref{u.l.term} is on that 
branch. It is also easy to see that if $\delta$ is not on the branch
$[0,\delta^+]_\mcT$ then there is a sharp for an inner model with a strong
cardinal, but we do not know if a substantially stronger large cardinal
hypothesis can be extrected in this situation. (For instance an inner model
with two strong cardinals.) 

\subsection{No inner models of ${\mbfM}$ with Woodin cardinals}

Finally, it is natural to wonder whether the anti-large cardinal assumption in
Theorem \ref{t0.1}(a) that there are no inner models with Woodin cardinals can
be localized to~${\mbfM}$. 

In this setting, we do not know if proper class extender models of $\mbfM$ are
iterable in $\mbfV$. We can however prove a weaker form of iterability for
such models, and establish a weak form of our main theorem. The problem is
that this weak form does not seem to suffice to deduce any interesting
applications, so we omit its proof. The precise statement of the result we
have in this case is as follows:  

\begin{theorem}
Assume that $\mbfM$ is a proper class inner model, and $\delta>\omega_1$ is a regular cardinal in 
$\mbfV$ such that $S_\delta$ is stationary. Granting that in $\mbfM$ there is no proper class inner 
model with a Woodin cardinal, the initial segment $\mbfK^{\mbfM}\cut\delta$ is universal for all iterable 
$1$-small premice in $\mbfV$ of cardinality less than $\delta$ in the following weak sense:

If $N$ is a $1$-small iterable premouse in $\mbfV$ of cardinality less than $\delta$, then there is a pair 
of iteration trees $(\mcT,\mcU)$ of length less than $\delta$, where $\mcT$ is on $\mbfK^{\mbfM}\cut
\delta$ and $\mcU$ is on $N$, such that $\mcU$ has a last model $M^\mcU_\infty$, there is no 
truncation on the main branch of $\mcU$, and one of the following holds: 

\begin{itemize} 
\item[(a)] 
The tree $\mcT$ has a last well-founded model $M^\mcT_\infty$, 
and $M^\mcU_\infty$ is an initial segment of $M^\mcT_\infty$. 
\item[(b)] 
The tree $\mcT$ has a limit length, does not have a cofinal 
well-founded branch, and $\delta(\mcT)=\delta(\mcU)$. Moreover, no ordinal larger than $\delta(\mcU)$ 
indexes an extender on $M^\mcU_\infty$ and, in some generic extension of $\mbfV$, there is a cofinal 
branch $b$ through $\mcT$ such that the direct limit along $b$ has well-founded part of length at least 
$\On\cap M^\mcU_\infty$. 
\end{itemize} 
\end{theorem}

The interest here, of course, would be to replace this weak universality with its genuine version or, if this 
is not possible, to see whether this version can actually be useful. For example, this version seems 
irrelevant for $\Sigma^1_3$-correctness, for which one needs to assume that there are no inner models
with a Woodin cardinal {\em in ${\mbfV}$}.


\end{document}